\documentclass[leqno]{article}
\usepackage{graphicx}
\usepackage{amsmath,amstext,amssymb,amsthm}
\usepackage{color}
\newtheorem{mythm}{Theorem}[section]
\newtheorem{mylem}{Lemma}[section]

\newtheorem{myrem}{Remark}[section]

\newtheorem{mypro}{Proposition}[section]

\newcommand{\norm}[1]{\left\Vert#1\right\Vert}
\newcommand{\norml}[2]{\left\Vert#1\right\Vert_{L^2(#2)}}

\newcommand{\abs}[1]{\left\vert#1\right\vert}

\newcommand{\set}[1]{\left\{#1\right\}}

\newcommand{\jm}[1]{\left[#1\right]}

\newcommand{\db}{\displaybreak[0]}
\newcommand{\nn}{\nonumber}

\newcommand{\de}{\delta}

\newcommand{\ga}{\gamma}

\newcommand{\Om}{\Omega}

\newcommand{\ta}{\theta}

\newcommand{\dx}{\,\mathrm{d} x}
\newcommand{\ds}{\,\mathrm{d} s}
\newcommand{\M}{\mathcal{M}}

\renewcommand{\i}{{\rm\mathbf i}}

\title{Continuous Interior Penalty Finite Element Method for Helmholtz Equation with High Wave Number: One Dimensional Analysis}
\author{
Lingxue Zhu
\thanks{Department of Mathematics, Nanjing University, Jiangsu,
210093, P.R. China. ({\tt zhulingxue@163.com}).  The work of the first author was
partially supported by  the NSF of China grants 11071116.}
\and
Erik Burman 
\thanks{Department of Mathematics, University of Sussex, Brighton, UK-BN1 9QH, United Kingdom.
({\tt E.N.Burman@sussex.ac.uk}).}
\and
Haijun Wu
\thanks{Department of Mathematics, Nanjing University, Jiangsu,
210093, P.R. China. ({\tt hjw@nju.edu.cn}). The work of the third author was
partially supported by the National Magnetic Confinement Fusion Science Program under grant 2011GB105003 and by the NSF of China grants 10971096, 11071116, 91130004.}
}

\begin{document}
\date{}
\maketitle
\setcounter{page}{1}
\begin{abstract}
This paper addresses the properties of Continuous Interior
Penalty (CIP) finite element solutions for the Helmholtz equation. The
$h$-version of the CIP finite element method with piecewise linear
approximation is applied to a one-dimensional model problem.
We first show discrete well posedness and convergence results, using the
imaginary part of the stabilization operator, for the complex Helmholtz
equation.
Then we consider a method with real valued penalty parameter
and prove an error estimate of the discrete solution in the $H^1$-norm, as the sum of best approximation
plus a pollution term that is the order of the phase difference. It is
proved that the pollution can be eliminated by selecting the penalty parameter
appropriately. As a result of this analysis,
thorough and rigorous understanding of the error behavior throughout the
range of convergence is gained. Numerical results are presented that
show sharpness of the error estimates and highlight some phenomena
of the discrete solution behavior.

\end{abstract}

\large




\section{Introduction}\label{sec-1}
The numerical solution of Helmholtz equation using the finite
element method (FEM) in the medium to high wave number remains a challenge
due to the strong pollution effects that are present in this regime.
It is known that when the standard Galerkin method is used a so
called scale resolution condition must be satisfied (see
\cite{MS11}) in order to achieve a quasi optimality estimate that
is robust in the wave number $k$. Invertibility of the linear system
also holds only under certain conditions on the relation between $k$
and the discretization parameters $h$ and $p$. This in particular
imposes the use of high order finite elements and seems to exclude
the possibility of using the simplest choice of piecewise affine
elements. In this latter case the standard Galerkin finite element
method has to be modified in order to obtain an efficient method.
Such modifications often takes the form of least squares terms
giving additional control of certain residual quantities, either in
the element or on element faces. For low order finite elements there
are a number of works on stabilized methods, typically using
Galerkin least squares approaches and some results on the effect of
the stabilization on the dispersion error exist in the one
dimensional case, see \cite{I.H.T}, or for an early example of the
use of face based residuals see \cite{O.P.}. Another possibility
is to use discontinuous Galerkin methods and in this framework it
has been proven by Feng and Wu \cite{F.W.} that provided a penalty
on the jumps of derivatives over element faces is added to the
formulation the linear system is always invertible. Similar results
were obtained using the continuous interior penalty method in a
recent work by Wu \cite{H.Wu} and numerical investigations showed
that the pollution error could be greatly reduced by choosing the
stabilization parameter appropriately. For wave-number-explicit error analyses of other methods including spectral methods and discontinuous Petrov-Galerkin methods, we refer to \cite{sw07,dpg11}.

In the present work we continue the investigations initiated in
\cite{H.Wu}, this time focusing on the one dimensional case and the
effect of the penalty operator on the errors in amplitude and phase.
Throughout the paper, $C$ is used to denote a generic positive constant which is independent of $k$, $h$, $f$. $C$ may have different values in different occurrences. We also use the shorthand notation $A \lesssim B$ and $B \lesssim A$ for the inequality $A \leq CB$ and $B \leq CA $. $A\simeq B$ is for the statement $A \lesssim B$ and $B \lesssim A$.
First we will
give alternative proofs of some of the results given in \cite{H.Wu},
showing for methods using a stabilization parameter with non-zero
imaginary part the linear system is always well posed and the
following error estimate holds
\[
\| (u - u_h)'\|\lesssim (kh +\min(1, k^3h^2))  \|f\|,
\]
where $\|\cdot\|$ denotes the $L^2$-norm. Then we consider the case
when the stabilization parameter is real and by constructing the
discrete Green's function we derive an error estimate where the error
is written as the sum of the best approximation error and a term
proportional to the phase error. We prove a relation between the phase
error and the stabilization parameter and show that for a particular
range of values for the stabilization parameter, under a mild
condition on the computational mesh, the pollution error is
eliminated, leading to the optimal error estimate
\[
\| (u - u_h)'\|\lesssim kh  \|f\|.
\]
These results are finally verified computationally in several numerical examples. 

This paper is organized as follows. In Section~2 we study the one-dimensional model problem and introduce the CIP-FEM. Pre-asymptotic error estimates in $H^1$- and $L^2$-norms are derived in Section~3 for any $k>0$, $h>0$ and imaginary penalty parameters. In Section~4, we consider the dispersion analysis of the CIP method and obtain the phase error estimates between the wave number $k$ of the continuous problem and some discrete wave number $k_h^-$ for different real penalty parameters. The discrete global system was solved explicitly in Section~5 via the theory of fundamental system, it plays a major part in the stability and pre-asymptotic error analysis. In Section~6, the stability and error estimates are proved directly and we can choose appropriate penalty parameter to eliminate the pollution effect in this section. Extensive numerical tests are given in Section~7 to show some phenomena of the discrete solution behavior and verify the theoretical findings, and we come to the conclusion in Section~8.
\section{The model problem and its discretization}\label{sec-2}

\subsection{The Boundary Value Problem}
Let $\Om=(0,1)$ and let on $\bar \Om$ the boundary value problem (BVP) $Lu=-f$ on be given:
\begin{align}
         u''(x)+k^2u(x) &=- f(x),\quad  x\in\Omega\label{e2.1a}\\
                  u(0)&=0,                      \label{e2.1b}  \\
                  u'(1)-\i ku(1)&=0,\label{e2.1c}
\end{align}
where, for simplicity, $f(x)\in L^2(\Omega)$ and $k$ is known as the wave number. We assume that $k\gg1$ since we are considering high-frequency problems.\\

\textbf{Notation}\\

 By $L^2(\Omega):=H^0(\Omega)$, we denote the space of all square-integrable complex-valued functions equipped with the inner
product
\[(v,w):=\int_\Om v(x)\bar w(x)\dx \text { and the norm } \norm{w}:=\sqrt{(w,w)}.\]
We use the notation $H^s(\Om)$ the Sobolev spaces of (integer) order $s$ in the usual sense.
Let $\norm{\cdot}_s$ and $\abs{\cdot}_s$ denote the usual full norm and seminorm on $H^s$, respectively.\\

\textbf{Existence and Uniqueness in $H^2(0,1)$}\\

The BVP \eqref{e2.1a}-\eqref{e2.1c} has a unique solution in the space $H^2(0,1)$. For the proof see, e.g., \cite{A.K.A}. The existence of the solution is concluded from the following construction.\\

\textbf{Inverse Operator}\\

The Green's function of the BVP \eqref{e2.1a}--\eqref{e2.1c} is
\begin{equation}\label{eG}
G(x,s)=\frac{1}{k}\left\{ \begin{aligned}
\sin{kx}\,e^{\i ks}, \quad 0\leq x\leq s, \\
\sin{ks}\,e^{\i kx}, \quad s\leq x\leq 1.
\end{aligned} \right.
\end{equation}
The solution $u(x)$ of \eqref{e2.1a}--\eqref{e2.1c} exists for all $k>0$ and can be written as
\begin{equation}\label{eu}
    u(x)=\int_0^1G(x,s)f(s)\ds,
\end{equation}
and we have,
\begin{equation}\label{eu'H}
    u'(x)=\int_0^1H(x,s)f(s)\ds \quad\text{ where }H(x,s)=\left\{
\begin{aligned}
&\cos{kx}\,e^{\i ks},&\quad 0\leq x < s, \\
&\i\sin{ks}\,e^{\i kx},&\quad s< x\leq 1.
\end{aligned}\right.
\end{equation}

\begin{mylem}\label{l2.1}
The BVP \eqref{e2.1a}--\eqref{e2.1c} has a unique solution in
$H^2(0,1)$ and for $f\in L^2(0,1)$
\begin{align}
    \norm{u}\le&  k^{-1}\norm{f},\label{e2.2a}\\
    \abs{u}_1\le&  \norm{f},\label{e2.2b}\\
    \abs{u}_2\le & (1+k)\norm{f}.\label{e2.2c}
\end{align}
\end{mylem}
\begin{proof}
See Douglas \textit{et al.} \cite{J.D.J}.
\end{proof}
\begin{myrem}
The aforementioned results are valid also for the adjoint problem \eqref{e2.1a}, \eqref{e2.1b} and
$ u'(1)+\i k u(1)=0.$
\end{myrem}
\subsection{The Continuous Interior Penalty method}
Let $\M_h$ be  a uniform mesh on $\bar{\Omega}$ that consists of $n$
sub-intervals $K_j=(x_{j-1}, x_j)$, $1\leq j \leq n$, where
$x_j=j/n$. Note that $x_j,1\le j\le n-1$ are interior nodes and
$x_0$ is the Dirichlet boundary node. The stepsize is $h = 1/n$. For
the ease of presentation, we assume that $k$ is constant on $\Om$.

For any function $v$, denote by $v_j^+=v(x_j+)$ and $v_j^-=v(x_j-)$
if the one-sided limits exist. We also define the jump $\jm{v}_j$ of
$v$ at a node $x_j$ as
\[\jm{v}_j:= v_j^--v_j^+, \qquad 1\le j\le n-1.\]
Now we define the ``energy'' space $V$ and the sesquilinear form $a_h(\cdot,\cdot)$ on $V \times
V$ as follows:
$$V:=\{v\in H^1(\Om)\wedge v(0)=0\}\cap \prod_{K_j \in \M_h}H^2(K_j),\ \ j=1,2,\cdots,n,$$
\begin{equation}\label{eah}
a_h(u,v):=(u',v')-k^2(u,v)-\i k u(1)\bar{v}(1)+J(u,v)\ \ \ \forall u,v \in V,
\end{equation}
where
\begin{equation}\label{eJ}
J(u,v):=\sum^{n-1}_{j=1}  \ga h[u']_j[\bar{v}']_j + \ga h (u'(1)
- \i k u(1)) (\bar v'(1)
- \i k \bar v(1))
\end{equation}
and $\ga:= \ga_{\tt Re} + i \ga_{\tt Im}$ is a complex number.
\begin{myrem}
(a) The terms in $J(u,v)$ are so-called penalty terms. The penalty
parameter in $J(u,v)$ is $\ga$.

(b) Penalizing the jumps of normal derivatives was used early by
Douglas and Dupont \cite{J.D.J.T} for second order PDEs and by
Babu\v{s}ka and Zl\'{a}mal \cite{I.B.M} for fourth order PDEs in the
context of $C^0$ finite element methods, by Baker \cite{G.A.B} for
fourth order PDEs and by Arnold \cite{D.A} for second order
parabolic PDEs in the context of IPDG methods. More recently it has
been proposed and analysed for fourth order PDEs by Hughes et al \cite{E.G.H.L.M.T}
and for singularly perturbed elliptic or parabolic problems by Burman
and co-workers \cite{Bu05,B.F,B.H}.

(c) Notice that we here add a least squares penalty on the boundary
condition as well. This enhances the continuity of the bilinear form
and appears to be necessary for the a priori error estimate proposed below.
\end{myrem}

It is clear that $J(u,v)=0$ if $u \in H^2(\Om)$ is the solution of \eqref{e2.1a}-\eqref{e2.1c} and $v\in V$. Therefore,
\begin{equation}
a_h(u,v)=(f,v), \ \ \ \forall v\in V.
\end{equation}
Let $V_h$ be the linear finite element space, that is,
\[V_h:=\set{v_h\in H^1(\Om):\; v_h(0)=0, v_h|_{K_j}\text{ is a linear polynomial, } j=1,\cdots,n}.\]
Then our CIP-FEMs are defined as follows: Find $u_h\in V_h$ such that
\begin{equation}\label{eipdg}
a_h(u_h,v_h)=(f,v_h)\qquad \forall v_h\in V_h.
\end{equation}
We remark that if the parameter $\ga\equiv 0$, then the above CIP-FEM becomes the standard FEM.

The following semi-norm on the space $V$ is useful for the subsequent analysis:
\begin{align}
   \norm{v}_{1,h}:=\Big(\|v'\|^2+\sum_{j=1}^{n-1}|\ga| h\abs{[v']_j}^2\Big)^{\frac12}. \label{enorm2}
\end{align}
\section{A priori error estimate for the model
  problem}
In this section we will use techniques similar to those developed in \cite{Bu05} to
derive an a priori error estimate that holds without any conditions on
the mesh parameter and the wave number. We present the analysis in the one dimensional case, but the extension to higher dimensions is straightforward. The key observations are
\begin{enumerate}
\item  if the complex component of
the stabilization coefficient is strictly negative (or positive
depending on the sign of the boundary condition), the formulation is
coercive on the stabilization;
\item if the $L^2$-projection is
used for interpolation in the analysis, the zeroth order term vanishes
and the bilinear form $a_h(\cdot,\cdot)$ has enhanced continuity
properties.
\end{enumerate}
These two observations lead to an a priori error estimate on the
stabilization operator that is optimal in $h$. An energy norm approach
combined with a duality argument is
then used to derive an a priori error
estimate of the error in the energy norm. To simplify the notation in
this section we assume that $\ga :=\i \ga_{\tt Im}$ the extension to
non-zero real part is straightforward.

Let $\pi_h:L^2(\Omega) \mapsto V_h$ be the standard $L^2$-projection
on $V_h$. It is straightforward to show that
\begin{equation}\label{L2err1}
\|u - \pi_h u\| + h \|\nabla (u - \pi_h u)\| \lesssim h^2 |u|_2
\end{equation}
and
\begin{equation}\label{L2err2}
|J(u -\pi_h u,u-\pi_h u)|^{\frac{1}{2}}
 \lesssim |\ga|^{\frac{1}{2}}(1+kh) h |u|_2,\ \ \Big(h^{-1}\sum\limits_{j=1}^{n}|(u-\pi_h u
 )(x_j)|^2\Big)^{\frac{1}{2}}\lesssim h|u|_2.
\end{equation}
In the following we will assume that $kh \lesssim 1$ and neglect high
order contributions in $kh$ in the above approximation estimates. We
first prove the continuity of $a_h(\cdot,\cdot)$ on the space
orthogonal to $V_h$. Let
\[
V^\perp := \{v \in V: (v,w_h) = 0, \forall w_h \in V_h\}.
\]
\begin{mylem}\label{perpcont}
For all $v \in V^\perp$ and all $w_h \in V_h$ there holds
\[
|a_h(v,w_h)| \lesssim
\Big(|J(v,v)|^{\frac{1}{2}}+|\ga|^{-\frac{1}{2}}\Big(\sum\limits_{j=1}^n
h^{-1}|v(x_j)|^2\Big)^{\frac{1}{2}}\Big)|J(w_h, w_h)|^{\frac{1}{2}}.
\]
\end{mylem}
\begin{proof}
The proof follows by observing that
\[
a_h(v,w_h) = (v',w_h') - \i k v(1) \overline{w_h}(1) + J(v,w_h).
\]
Noting that $w_h$ is piecewise linear and after an integration by parts in the first term in the right hand side
we have
\[
a_h(v,w_h) = \sum_{j=1}^{n-1} v(x_j) [\overline{w_h}']_j +   v(1) (-\i k \overline{w_h}(1)+\overline{w_h}'(1))
+ J(v,w_h).
\]
We conclude by applying the Cauchy-Schwarz inequality.
\end{proof}

For the stabilization operator $J(\cdot,\cdot)$ we have the
following stability estimate.
\begin{mylem}\label{Jstab}
Assume that $\gamma_{\tt Im} < 0$. For all $v_h \in V_h$ there holds
\[
|J(v_h, v_h)| + k|v_h(1)|^2 = -{\textnormal {Im}}[a_h(v_h,v_h)]
\]
and for $u_h$ solution to \eqref{eipdg} then
\[
| J(u_h, u_h)| + k|u_h(1)|^2 = -{\textnormal {Im}}[(f,u_h)].
\]
\end{mylem}
\begin{proof}
Immediate by the definition of $a_h(\cdot,\cdot)$ and \eqref{eipdg}.
\end{proof}
\begin{myrem}\label{existsol}
For all $\gamma_{\tt Im}<0$ Lemma \ref{Jstab}, implies existence of a
unique discrete solution, since $|J(v_h, v_h)| + k|v_h(1)|^2$ is
a norm on $V_h$.
\end{myrem}

Combining the two previous results with the consistency of the formulation and the regularity estimate \eqref{e2.2c} immediately
gives us a convergence estimate for the penalty term $J(\cdot,\cdot)$
and the error in the right end point.
\begin{mypro}\label{Jconv}
Let $u\in H^2(\Omega)$ be the solution of \eqref{e2.1a}-\eqref{e2.1c} and $u_h \in V_h$ be
the solution of \eqref{eipdg}. Then there holds
\[
|J(u-u_h, u - u_h)|^{\frac{1}{2}} + k^{\frac{1}{2}}|(u - u_h)(1)|
\lesssim
  \big(|\ga|^{\frac{1}{2}}(1+ kh)+|\ga|^{-\frac{1}{2}}\big) k h \|f\|.
\]
\end{mypro}
\begin{proof}
Let $u - u_h = \eta - \xi_h$ with $\eta = u - \pi_h u$ and $\xi_h =
u_h - \pi_h u$. By the triangle inequality and the error estimate
\eqref{L2err2} it is enough to consider
$$
|J(\xi_h,  \xi_h)|^{\frac{1}{2}}+ k^{\frac{1}{2}}|\xi_h(1)|.
$$
Using Lemma \ref{Jstab} followed by the consistency we have
\[
|J(\xi_h,\xi_h)| + k|\xi_h(1)|^2 =  -{\textnormal {Im}}[a_h(\xi_h, \xi_h)]=-{\textnormal {Im}}[a_h(\eta, \xi_h)]\le |a_h(\eta,\xi_h)|.
\]
We then apply the continuity of Lemma \ref{perpcont} to bound the
right hand side,
\[| J(\xi_h,\xi_h) |+ k|\xi_h(1)|^2 \lesssim
 \Big(|J(\eta,\eta)|^{\frac{1}{2}}+|\ga|^{-\frac{1}{2}}\Big(\sum\limits_{j=1}^n
h^{-1}|\eta(x_j)|^2\Big)^{\frac{1}{2}}\Big) |J(\xi_h,
\xi_h)|^{\frac12}.\] Hence,
\begin{equation}\label{aa7}
 |J(\xi_h, \xi_h)|^{\frac{1}{2}} +
k^{\frac{1}{2}}|\xi_h(1)| \lesssim
 |J(\eta,\eta)|^{\frac{1}{2}}+|\ga|^{-\frac{1}{2}}\Big(\sum\limits_{j=1}^n
h^{-1}|\eta(x_j)|^2\Big)^{\frac{1}{2}},
\end{equation}
then the claim follows by applying once again \eqref{L2err2}. The
proof is completed.
\end{proof}

After these preliminary results we are in a position to prove the main
result of this section.
\begin{mythm}(A priori error estimates)\label{imge th}\\
Let $u\in H^2(\Omega)$ be the solution of
\eqref{e2.1a}-\eqref{e2.1c} and $u_h \in V_h$ the solution of
\eqref{eipdg}, with $\gamma_{\tt Im}<0$. Then, if $h$ is small such
that $k h \lesssim 1$ for all $h>0$ and $k\geq 1$, there
holds
\[
\|k(u - u_h)\| \lesssim (|\ga|+|\ga|^{-1}) \min(1,k^3 h^2) \|f\|
\]
and
\[
\|(u - u_h)'\| \lesssim(|\ga|+|\ga|^{-1}) \big(kh + \min(1,k^3 h^2)\big)
\|f\|.
\]
\end{mythm}
\begin{proof}
Using once again the decomposition $u - u_h = \eta -\xi_h$, by the
estimate \eqref{L2err1}, we only need to estimate the error in
$\xi_h$. Consider the adjoint problem, find $z \in H^2(\Om)$ such
that
\begin{equation}\label{adjoint}
(w',z') - k^2(w,z) - \i k w(1) \bar z(1)  = (w,\xi_h)\quad \forall w\in V
\end{equation}
and its finite element equivalent, find $z_h \in V_h$ such that
\begin{equation}\label{adjointFEM}
a_h(w_h,z_h) = (w_h,u_h-\pi_h u)\quad \forall w_h\in V_h.
\end{equation}
By Lemma \ref{Jstab} and Proposition \ref{Jconv}, $z_h$ exists and satisfies
\[
|J(z_h,z_h)|^{\frac{1}{2}}= |J(z-z_h,z-z_h)|^{\frac12}\lesssim
(|\gamma|^{\frac{1}{2}}+|\ga|^{-\frac{1}{2}}) kh \|\xi_h\|.
\]
Using the consistency of the formulation and the continuity of
Lemma~\ref{perpcont} followed by the \eqref{L2err2} we get
\begin{align*}
\|\xi_h\|^2& = a_h(\xi_h,z_h) = a_h(\eta,z_h) \\
& \lesssim
\Big(|J(\eta,\eta)|^{\frac{1}{2}}+|\ga|^{-\frac{1}{2}}\Big(\sum\limits_{j=1}^n
h^{-1}|\eta(x_j)|^2\Big)^{\frac{1}{2}}\Big)
|J(z_h,z_h)|^{\frac{1}{2}}\\
& \lesssim (|\ga|+|\ga|^{-1})(kh)^2 \|f\| \|\xi_h\|.
\end{align*}
Therefore,
\begin{equation}\label{aa8}
  \|k\xi_h\|\lesssim (|\ga|+|\ga|^{-1})k^3h^2 \|f\|.
\end{equation}
Next we show that $\|k\xi_h\|\lesssim (|\ga|+|\ga|^{-1}) \|f\|$. In fact,
 it follows from the
definition of the sesquilinear form $a_h(\cdot,\cdot)$ that
\begin{align*}
 \|k \xi_h\|^2 &= -{\textnormal{Re}}[a_h(\xi_h,\xi_h)] + (\xi_h',\xi_h')\\
 & \lesssim
 |a_h(\eta,\xi_h)| + (J(\xi_h,\xi_h) + k |\xi_h(1)|^2)^{\frac{1}{2}}(|\ga|^{-\frac{1}{2}}(k h)^{-1}+(k h)^{-\frac{1}{2}}) \|k
 \xi_h\|,
\end{align*}
where we have used an integration by parts in the second term in the
right hand, i.e., $(\xi_h',\xi_h')=\sum_{j=1}^{n-1}[\xi_h']_j\overline{\xi_{h}}(x_j)+(\xi_h'(1)-\i k\xi_h(1))\overline{\xi_h}(1)+\i k|\xi_h(1)|^2$, to derive the last inequality.\\
From the continuity of Lemma~\ref{perpcont} and \eqref{L2err2},
\eqref{aa7} we conclude that
\begin{align*}
|a_h(\eta,\xi_h)| &\lesssim
\Big(|J(\eta,\eta)|^{\frac{1}{2}}+|\ga|^{-\frac{1}{2}}\Big(\sum\limits_{j=1}^n
h^{-1}|\eta(x_j)|^2\Big)^{\frac{1}{2}}\Big) |J(\xi_h,
\xi_h)|^{\frac12}\\
 &\lesssim (|\ga|+|\ga|^{-1})(kh)^2 \|f\|^2.
\end{align*}
Therefore,
\begin{align*}
\|k \xi_h\|^2 &\lesssim (|\ga|+|\ga|^{-1}) (kh)^2\|f\|^2 +
(J(\xi_h,\xi_h) + k |\xi_h(1)|^2)(|\ga|^{-1}(k h)^{-2}+(k h)^{-1})\\
&\lesssim (1+|\ga|+|\ga|^{-2})(1+ (kh)^2)\|f\|^2,
\end{align*}
which together with \eqref{aa8} proves the first claim.

By the definition of $a_h(\cdot,\cdot)$ once again and Galerkin
orthogonality there holds
\begin{align*}
\|\xi_h'\|^2 &= {\textnormal{Re}}[a_h(\xi_h,\xi_h)] + \|k \xi_h\|^2
\lesssim |a_h(\eta,\xi_h)| + \big((|\ga|+|\ga|^{-1})\min(1,k^3h^2)
\|f\|\big)^2.\\
&\lesssim (|\ga|+|\ga|^{-1})(kh)^2 \|f\|^2 +
\big((|\ga|+|\ga|^{-1})\min(1,k^3h^2) \|f\|\big)^2.
\end{align*}
 That is, the second claim holds. This completes the proof of the theorem.
\end{proof}
\begin{myrem}
(a) Note that the above estimate does not impose any constraints on
the choice of the mesh size $h$ compared to $k$. Both estimates
exhibit the standard pollution term, but nevertheless the errors are
upper bounded by data, independently of $h$ and $k$. This shows that
the imaginary part of the stabilization gives control of the
amplitude of the wave.

(b) If the penalty term on the boundary
condition is removed, i.e., if $J(u,v)$ in \eqref{eJ} is replaced by $J(u,v):=\sum\limits_{j=1}^{n-1}\ga h[u']_j[\bar{v}']_j$
then  Theorem~\ref{imge th} still holds. This can be proved by following the analysis given in \cite{H.Wu}. We omit the details.
 As we shall see in the
next section, the real part of the stabilization allows us to
control the phase error provided the stabilization parameter is
chosen appropriately.
%
\end{myrem}
\section{Dispersion analysis}\label{sec-3}
In this section we will consider the case where $\gamma$ is a real
number. Using a dispersion analysis we will derive precise bounds on
the error in the numerical wavenumber. These bounds are then used to
prove that a particular choice of the penalty parameter allows to eliminate the pollution in the one dimensional case.
\subsection{Global FE-equations and discrete fundamental system}
Let $\{\phi_1, \phi_2, \cdots,
\phi_{n-1},\phi_n\}$ be the nodal
basis functions for the space $V_h$ satisfying $\phi_j(x_l)=\delta_{jl}, $ the Kronecker delta, for $ j=1, 2, \cdots, n$ and $l=0, 1, \cdots, n.$
Then the CIP-FEM solution can be spanned as:
\[u_h(x)=\sum_{j=1}^n u_{h,j} \phi_j \ \ \ \textrm{with}\ \   u_{h,j}=u_h(x_j),\ \ j=1,2,\cdots,n.\]
Let $v_h=\phi_i, i=1,\cdots,n$ in \eqref{eipdg}, the CIP formulation can be rewritten as the following
linear system:
\begin{equation}\label{eipdgM}
    L_hU=h F,
\end{equation}
where
\begin{equation*}
   L_h\hskip -2pt =\hskip -2pth\begin{pmatrix}
      a_h(\phi_j,\phi_i)  \\
    \end{pmatrix}_{n\times n},
    U\hskip -2pt=\hskip -2pt\begin{pmatrix}
        u_{h,i}
      \end{pmatrix}_{n\times 1},
      F\hskip -2pt=\hskip -2pt\begin{pmatrix}
          (f, \phi_i)
        \end{pmatrix}_{n\times 1}.
\end{equation*}
 Denote by $t=kh$, $R=-1-4\gamma-t^2/6$, $S=1+3\gamma-t^2/3$, we have

\begin{equation}\label{eLh1}
L_h=
  \begin{pmatrix}
   2S-\gamma & R & \gamma &  &  &  &          \\
   R & 2S & R & \gamma &  &  &         \\
   \gamma & R & 2S & R & \gamma &  &         \\
     &  & \ddots &   \ddots& \ddots& &      \\
            &  & \gamma & R & 2S & R & \gamma \\
            &  & & \gamma &  R & 2S-\gamma & R+2\gamma\\
            &  & & & \gamma & R+2\gamma & S-2\gamma-\i t  \\
  \end{pmatrix}.
\end{equation}
\begin{myrem}
The product $t=kh$ is a measure of the number of elements per wavelength (of
the exact solution). In particular, if the stepwidth is such that $t=\frac{\pi}{l}$ for integer $l$ then exactly
$l$ elements are placed on one half-wave of the exact solution.
\end{myrem}

\subsection{Discrete wavenumber and Dispersion analysis}
Recall that $k$ is the wave number for the BVP
\eqref{e2.1a}--\eqref{e2.1c} and that the functions $e^{\pm\i kx}$
play an important role in the solution of the BVP which satisfy the
equation \eqref{e2.1a} with $f=0$. The discrete wave number $k_h$
for the CIP method is defined similarly by considering the vector
$v$ with $v_j=e^{\i k_h j h}$ and solving the following  ``interior''
equations:
\begin{equation}\label{e3.1}
   \gamma v_{j-2}+ R v_{j-1} + 2S v_j + R v_{j+1} +\gamma v_{j+2}=0, \; j=3,\cdots,n-2.
\end{equation}
Denote by $t_h=k_hh$, the above equations are equivalent to the
equation
\begin{equation}\label{e3.2}
    2\gamma \cos^2 t_h - \left(
 4\gamma+1+\frac{t^2}{6}\right) \cos t_h + 2
\gamma+1-  \frac{t^2}{3} =0,
\end{equation}
which has the roots
\begin{align}\label{e3.3}
\cos t_h^\pm = \frac{4\gamma+1+\frac{t^2}{6} \pm
\sqrt{\left(1+\frac{t^2}{6}\right)^2+4\gamma t^2} }{4 \gamma }.
\end{align}
Some simple calculations show that $|\cos t_h^{-}|\leq 1\le |\cos
t_h^{+}|$ if $\ga\geq
-1/4+t^2/48$ and $|\cos t_h^{-}|\ge 1\ge |\cos
t_h^{+}|$ otherwise. Without
loss of generality, assume $|\cos t_h^{-}|\leq 1$, and define $k_h^{-} := t_h^{-}/h$ and $k_h^+:=t_h^+/h$.
Noting that a large $|\ga|$ may cause a large error (cf.  Theorem~\ref{imge th}) and that $\cos t_h^+$ can not approximate $\cos t$ well ($\cos t_h^+=\frac{2\ga+1}{2\ga}\neq 1$ at $t=0$), for simplicity, in the following we will assume that
$-{1}/{6}\leq\gamma \leq 1/6$. Physically, case $(-)$ describes a
propagating wave whereas case $(+)$ describes a decaying wave
\cite{I.H.T}.
\begin{mylem}\label{lkh3}
Assume that $t=k h\le 1$, $-\frac{1}{6}\leq\gamma \leq \frac{1}{6}$,
then we may show\\
 {\rm(i)} $ \abs{\cos t_h^{-}-1+\frac{t^2}{2}}\le
\frac{1}{6}  t^4, \ \
 \abs{k_h^{-} - k }\lesssim  k^3h^2;$ \\
{\rm(ii)} $\text{If }  \gamma=-1/12,\  \text{then } \abs{k_h^{-} - k }\lesssim
k^5h^4;$ \\
{\rm(iii)}$ \text{If } \abs{\gamma-\gamma_o} \lesssim \frac{1}{k^2 h} \text{
where }\gamma_{o}=\dfrac{6\cos t-6 + t^2\cos t+ 2t^2}{12(1-\cos
t)^2},\text{ then }\abs{k_h^{-}- k}\lesssim k h.$
\end{mylem}

\begin{proof}
 Denote $t_h=t_h^{-}$ and from \eqref{e3.3}, we have
\begin{equation}\label{ac}
{1-\cos t_h^{-}}={\frac{t^2}{1+\frac{t^2}{6}+\sqrt{\big(1+\frac{t^2}{6}\big)^2+4\ga t^2}}}\leq \frac{t^2}{2}
\end{equation}
and
\begin{align}\label{ab}
\sqrt{\big(1+\frac{t^2}{6}\big)^2+4\ga t^2}=4\ga (1-\cos
t_h^{-})+1+\frac{t^2}{6}.
\end{align}
Clearly,  $t_h^{-}=k_h^{-} h \in (0,\frac{\pi}{2})$ (cf. \eqref{ac}).
It follows from  \eqref{e3.3} and \eqref{ab} that
\begin{align}\label{g3.5}
\cos
t_h^{-}-1+\frac{t^2}{2}&=\frac{(\frac{2}{3}+4\ga)t^4}{2(1+\frac{t^2}{6}+\sqrt{\big(1+\frac{t^2}{6}\big)^2+4\ga
t^2})(1-\frac{t^2}{6}+
\sqrt{\big(1+\frac{t^2}{6}\big)^2+4\ga t^2})}\\
\nn &=\frac{(\frac{1}{3}+2\ga)t^4}{2+\frac{t^2}{3}+4\ga
t^2+2\sqrt{\big(1+\frac{t^2}{6}\big)^2+4\ga t^2}}\\
\nn &=\frac{(1+6\ga)t^4}{2(6+t^2+12\ga(1-\cos t_h^{-})+ 6\ga t^2)}\\
\nn &=\frac{\left(1+6\gamma\right)t^4}{12} -
\frac{\left(1+6\gamma\right)\left(1+ 6\gamma +12\gamma(1-\cos
t_h^{-})/t^2\right)t^6}{12\left(6+t^2 +12\gamma(1-\cos t_h^{-})+ 6\gamma
t^2\right)},\
\end{align}
which together with \eqref{ac} implies the first inequality of $\rm(i)$. The second
inequality and \rm(ii) can be proved easily as follows: the
inequality $\sin\ta>\frac2{\pi}\ta,$ $\forall
\ta \in(0,\frac{\pi}{2})$ implies
\begin{equation}\label{bb5}
t \abs{t_h^--t}\lesssim \abs{t_h^--t}\abs{t_h^-+t}\lesssim \abs{2\sin\frac{t_h^--t}{2}\sin\frac{t_h^-+t}{2}}
    =\abs{\cos t_h^--\cos t}
\end{equation}
and it is easy to show that:
\[\abs{\cos t_h^{-} -\left(1-\frac{t^2}{2} + \frac{6\gamma+1}{12}t^4\right)} \lesssim t^6,\qquad
\abs{\cos t -\left(1-\frac{t^2}{2} + \frac{t^4}{24}\right)} \lesssim
t^6,
\] which implies that the second inequality of \rm(i) and $\rm(ii)$ hold.

In the following, we turn to prove the last inequality. Note that
$\cos t_h^{-}$ is the function of $\gamma$ and $t_h^{-}$, $\gamma_o$
satisfies $\cos t_h^-(\gamma_o)=\cos t$ and hence
$t_h^{-}(\gamma_o)=t.$
 By some simple calculations,
\begin{align*}
    \abs{\cos t_h^{-}-\cos t}&=\abs{\frac{1+\frac{t^2}{6}-\sqrt{\left(1+\frac{t^2}{6}\right)^2+4\gamma t^2} }{4 \gamma
}-\frac{1+\frac{t^2}{6}-\sqrt{\left(1+\frac{t^2}{6}\right)^2+4\gamma_o
t^2} }{4 \gamma_o }}\\
&=\abs{\frac{\sqrt{\left(1+\frac{t^2}{6}\right)^2+4\gamma
t^2}-\sqrt{\left(1+\frac{t^2}{6}\right)^2+4\gamma_o
t^2}}{({1+\frac{t^2}{6}+\sqrt{\left(1+\frac{t^2}{6}\right)^2+4\gamma
t^2}})({1+\frac{t^2}{6}+\sqrt{\left(1+\frac{t^2}{6}\right)^2+4\gamma_o
t^2}})}}t^2\\
&\lesssim \abs{\sqrt{\left(1+\frac{t^2}{6}\right)^2+4\gamma
t^2}-\sqrt{\left(1+\frac{t^2}{6}\right)^2+4\gamma_o t^2}}t^2\\
&\lesssim\abs{\frac{4(\ga-\ga_o)}{\sqrt{\left(1+\frac{t^2}{6}\right)^2+4\gamma
t^2}+\sqrt{\left(1+\frac{t^2}{6}\right)^2+4\gamma_o t^2}}}t^4 \lesssim \abs{\gamma-\gamma_o}t^4,
\end{align*}
and \eqref{bb5} therefore,
\begin{align*}
    t \abs{t_h^{-}-t}\lesssim \abs{\cos t_h^{-}-\cos t}\lesssim t^4\abs{\gamma-\gamma_o} \lesssim t^2h,
\end{align*}
which implies that $\rm(iii)$ holds. This completes the proof of the lemma.
\end{proof}
\begin{myrem}
Note that the phase difference between the exact and the linear finite element solutions obtained is $O(k^{3}h^{2})$ (cf. \cite{M.A,F.I.I.B}). While for the CIP-FEM, if the penalty parameter $\ga$ is close enough to $\ga_o$ the phase difference is $O(kh)$ and, as a result, the CIP-FEM is pollution free (cf. Theorem~\ref{terr3} below). Figure~\ref{best penparam} gives a plot of the optimal penalty parameter $\ga_o$ versus $t$ for $0<t\le 1$.
\begin{figure}[!htb]
\begin{center}
\includegraphics[scale=0.6]{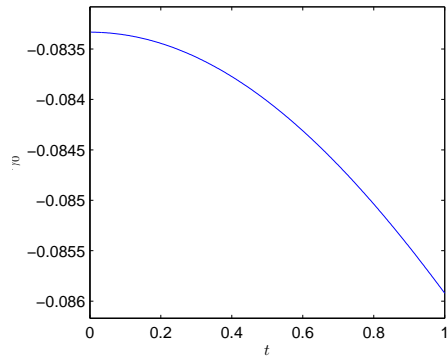}
\end{center}
\caption{\small{The optimal penalty parameter versus $t=kh\leq 1$.}}\label{best penparam}
\end{figure}

\end{myrem}
\section{The discrete Green's function}

To construct the discrete Green's function, we first find the
inverse of the stiffness matrix $L_h$. Inspired by the formulation
of the Green's function for the BVP (cf. \eqref{eG}), we find
$G_h=L_h^{-1}$ of the following form:
\begin{align}\label{eGh}
G_{h,j,m}=&\left\{ \begin{aligned}
&\textstyle\sum_ {i=1}^4 A_{m,i} \eta_i^j , &j< m,\\
&\textstyle\sum_{i=1}^4B_{m,i}\eta_i^j , &j\ge m, \end{aligned}
\right. \end{align} where$\quad \eta_1=e^{-\i k_h^{-}h},\eta_2=e^{\i
k_h^{-}h},\eta_3=e^{-\i k_h^{+}h},\eta_4=e^{\i k_h^{+}h}$.

By the definition of $\eta_i, i=1,2,3,4$ there holds the facts:
\begin{equation}\label{bbb1}
\eta_1\eta_2=\eta_3\eta_4=1,\qquad\eta_1+\eta_2=2\cos t_h^-,\qquad  \eta_3+\eta_4=2\cos t_h^+.
\end{equation}
If $|\ga|\leq 1/6$, by some simple calculations, we can get
\begin{equation}\label{bbb3}
|\cos t_h^+ -1|\geq 3.
\end{equation}
Without loss of generality, assume $|\eta_4|>|\eta_3|$, it is clear that
\begin{equation}\label{bbb2}
|\eta_4|>3 \ \ {\rm {and}} \ \  |\eta_3|<\frac{1}{3}.
\end{equation}
 From \eqref{eGh}, the solution of \eqref{eipdgM} is represented
as
\begin{align}\label{euh}
    u_{h,j}=h\sum_{m=1}^{n} G_{h,j,m}(f, \phi_{m}),\ \ \ j=1,2,\cdots,n,
\end{align}
and hence the CIP-FEM solution is given by
\begin{align*}
    u_h=\sum_{j=1}^n u_{h,j}\phi_j.
\end{align*}
To represent the derivative of the CIP-FEM solution, we define a $n
\times n$ matrix $H_h$ as
\begin{align}\label{eHh}
H_{h,j,m}=G_{h,j,m}-G_{h,j-1,m}\ \  1 \le j \le n,\ \ \text{Here}\
\ G_{h,0,m}:=0.
\end{align}
 It is
clear that
\begin{align}\label{euh'1}
    u_h'(x)=\frac{u_{h,j}-u_{h,j-1}}{h}=\sum_{m=1}^{n} H_{h,j,m}(f, \phi_m), \quad\forall x\in [x_{j-1},x_j], \; j=1,\cdots,n.
\end{align}

Throughout this section let $\widetilde{C}$ denote a \emph{general function} that may have different expressions at different places but is bounded (uniformly) by some constant independent of $k$, $h$, and the penalty parameters.
We
first state a simple but useful lemma without proof.
\begin{mylem}\label{LA1}Suppose $0< t \le 1$, if $|b|\le
\sigma_1 |a|$, $0<\sigma_1<1$, $a$, $b$ and $\sigma_1$ are independent of the penalty parameter. Then
\begin{align}\label{esti}
\frac{1}{a-bt}=\frac{1}{a}\left(1+\widetilde{C}t \right).
\end{align}
\end{mylem}

 The following lemma presents estimates for $H_{h,j,m}$.
\begin{mylem}\label{lHh2} Assume that $t=kh\le 1$, $k\ge 1$, $0 < |\gamma| \leq \frac{1}{6} $.  Then
\begin{eqnarray}
H_{h,j,m} =\left\{ \begin{aligned}
&\cos(j t_h^-)  e^{\i m t_h^-}+\widetilde{C} t +\widetilde{C} \eta_4^{j-m} , \quad&j < m,\\
&\i\sin(m t_h^-) e^{\i j t_h^-}+\widetilde{C} t +\widetilde{C}
\eta_4^{m-j} , \quad&j \ge m,
\end{aligned} \right.\label{eHh2}
\end{eqnarray}
where $\widetilde{C}$ is a general function which is bounded by some constant independent of $k$, $h$, and the penalty parameters.
\end{mylem}
\begin{proof} The proof is divided into four steps.

\textbf{Stpe 1. Solving for $A_{m,i}$ and $B_{m,i}$}.
$G_{h,j,m}$ are determined by the system of equations:
\begin{align}\label{acc1}
\left\{\begin{array}{ll}
    (2S-\gamma)G_{h,1,m}+ R G_{h,2,m}+ \gamma G_{h,3,m}=\de_{1,m}, \\
     R G_{h,1,m}+ 2SG_{h,2,m}+ R G_{h,3,m}+ \gamma G_{h,4,m}=\de_{2,m},  \\
   \gamma G_{h,n-3,m}+R G_{h,n-2,m} +(2S-\gamma) G_{h,n-1,m} +(R+2\gamma) G_{h,n,m}
    =\de_{n-1,m}, \\
      \gamma G_{h,n-2,m}+(R+2\gamma) G_{h,n-1,m} +(S-2\gamma-\i t) G_{h,n,m}
    =\de_{n,m},\\
    \gamma G_{h,j-2,m}+R G_{h,j-1,m} +2S G_{h,j,m} +R G_{h,j+1,m} +\gamma G_{h,j+2,m}
    =\de_{j,m},
 \end{array} \right.
\end{align}
where $3 \le j \le n-2$ in the last equality of the above system and
$\de_{j,m}$, $1\leq j, m \leq n$ are the Kronecker delta.

Formula \eqref{e3.1} yields
\begin{align}\label{acc2}
   & \gamma\eta_i^{-2}+R\eta_i^{-1}+ 2S
    + R \eta_i+\gamma\eta_i^2 =0.
\end{align}

We first consider $m=5,\cdots,n-3.$ From \eqref{eGh} and \eqref{acc2},
the system \eqref{acc1} is reduced to the following
system of eight equations:
\begin{align}\label{abc1}
\left\{\begin{array}{ll}
    \textstyle\sum_{i=1}^4\eta_i (2S-\gamma+ R \eta_i + \gamma\eta_i^2) A_{m,i}=0,\\
    \textstyle\sum_{i=1}^4 \eta_i^2(R\eta_i^{-1}+ 2S+ R \eta_i + \gamma\eta_i^2) A_{m,i}=0, \\
    \textstyle\sum_{i=1}^4 \eta_i^{m-2}\left[(\gamma\eta_i^{-2}+R\eta_i^{-1}+ 2S +
    R \eta_i) A_{m,i}+ (\gamma\eta_i^2) B_{m,i}\right]=0, \\
   \textstyle\sum_{i=1}^4\eta_i^{m-1}\left[(\gamma\eta_i^{-2}+R\eta_i^{-1}+ 2S
    ) A_{m,i}+ (R \eta_i+\gamma\eta_i^2)B_{m,i}\right]=0, \\
    \textstyle\sum_{i=1}^4\eta_i^{m}\left[(\gamma\eta_i^{-2}+R\eta_i^{-1}
    ) A_{m,i}+ (2S+R \eta_i+\gamma\eta_i^2)B_{m,i}\right]=1, \\
    \textstyle\sum_{i=1}^4\eta_i^{m+1}\left[(\gamma\eta_i^{-2}
    ) A_{m,i}+ (R\eta_i^{-1}+2S+R \eta_i+\gamma\eta_i^2)B_{m,i}\right]=0, \\
     \textstyle\sum_{i=1}^4\eta_i^{n-1}\big[\gamma \eta_i^{-2}+R \eta_i^{-1} + 2S-\gamma  +(R+ 2\gamma) \eta_i\big] B_{m,i}=0, \\
     \textstyle\sum_{i=1}^4\eta_i^{n}\big[\gamma \eta_i^{-2}+(R+2\gamma) \eta_i^{-1} + S-2\gamma-\i t \big] B_{m,i}=0.
 \end{array} \right.
\end{align}
Plugging \eqref{acc2} into the first seven equations of \eqref{abc1} gives
\begin{eqnarray}\label{abc2}
\begin{aligned}
    &\textstyle\sum_{i=1}^4(\gamma \eta_i^{-1} +R + \gamma \eta_i) A_{m,i}=0,\\
    &\textstyle\sum_{i=1}^4 \gamma A_{m,i}=0,  \\
    &\textstyle\sum_{i=1}^4 \gamma \eta_i^m (B_{m,i}- A_{m,i})=0, \\
    &\textstyle\sum_{i=1}^4  (R + \gamma\eta_i)\eta_i^m(B_{m,i}-A_{m,i})=0,\\
    &\textstyle\sum_{i=1}^4  (\ga \eta_i^{-2} + R \eta_i^{-1})\eta_i^m(A_{m,i}-B_{m,i})=1,\\
    &\textstyle\sum_{i=1}^4 \gamma \eta_i^{m-1 }(A_{m,i}- B_{m,i})=0, \\
    &\textstyle\sum_{i=1}^4 \ga (\eta_i^{-1} -2 + \eta_i)\eta_i^n B_{m,i}=0.
\end{aligned}
\end{eqnarray}
By $R=-1-4\gamma-t^2/6$, $S=1+3\gamma-t^2/3$, the eighth equation of
\eqref{abc1} yields
\begin{equation}\label{abc3}
 \sum_{i=1}^4 \Big(1-\dfrac{t^2}{3}-\Big(1+\dfrac{t^2}{6}\Big)\eta_i^{-1}+\ga(1-\eta_i^{-1})^2-\i
  t\Big)\eta_i^n B_{m,i}=0.
\end{equation}
Then, by simplifying \eqref{abc2} and \eqref{abc3}, a $8\times 8$
system which is equivalent to the system \eqref{abc1} can be obtained:
\begin{align}\label{eAB}
\left[ \begin {array}{cc} -U_m & U_m\\\noalign{\medskip} V_1&
V_2\end {array} \right] \left[
\begin {array}{c}
A_m\\\noalign{\medskip} B_m\end {array} \right]=\left[
\begin {array}{c} z \\\noalign{\medskip}0\end {array} \right]
\qquad z=\left[ -1/\gamma,0,0,0 \right]^T,
\end{align} where
$A_m=[A_{m,1},A_{m,2},A_{m,3},A_{m,4}]^T$,
$B_m=[B_{m,1},B_{m,2},B_{m,3},B_{m,4}]^T$, and the $i$-th $(i=1,2,3,4)$ column of the
matrix $U_m$, $V_1$, $V_2$ are stated as follows:
\begin{align*}
&U_m (:,i)=\eta_i^m\left( \begin
{array}{c}  \eta_i^{-2}\\
\noalign{\medskip}  \eta_i^{-1} \\
\noalign{\medskip} 1\\
\noalign{\medskip} \eta_i\end{array} \right) , \qquad V_1(:,i)=\left(
\begin {array}{c}
\eta_i^{-1} + \eta_i\\
\noalign{\medskip} 1 \\ \noalign{\medskip}
0  \\ \noalign{\medskip} 0  \end{array} \right ), \qquad
 V_2(:,i) =\left( \begin {array}{c} 0 \\
\noalign{\medskip} 0  \\
\noalign{\medskip} a_i\\
\noalign{\medskip} b_i
\end{array} \right),
\end{align*}
\begin{align}
 &a_i=( \eta_i^{-1}-2+
\eta_i )\eta_i^{n},\label{eai}\\
 &b_i=\Big( 1-\frac{t^2}{3}-\Big(1+\frac{t^2}{6} \Big)\eta_i^{-1}+
\gamma(1-\eta_i^{-1})^2-\i t \Big)\eta_i^{n},\qquad i=1,2,3,4.\label{ebi}
\end{align}

Next we consider $m=2,3,4,n-2,n-1,n$, there will be less than $8$
equations, that is, the linear system is underdetermined,  however, we can show that the system \eqref{eAB} gives a special solution. We only prove the case $m=2$, other cases ($m=3,4,n-2,n-1,n$) can be obtained similarly,
we leave the derivation to the interested reader. When $m=2$, from \eqref{eGh} and \eqref{acc2}, the system \eqref{acc1} is reduced to the following
system of five equations:
\begin{eqnarray}\label{abc4}
\left\{\begin{array}{ll}
    \textstyle\sum_{i=1}^4\eta_i \left[(2S-\gamma)A_{2,i}+ (R \eta_i + \gamma\eta_i^2) B_{2,i}\right]=0,\\
    \textstyle\sum_{i=1}^4 \eta_i^2\left[R\eta_i^{-1}A_{2,i}+( 2S+ R \eta_i + \gamma\eta_i^2) B_{2,i}\right]=1, \\
    \textstyle\sum_{i=1}^4\eta_i^3\left[(\gamma\eta_i^{-2}
    ) A_{2,i}+ (R\eta_i^{-1}+2S+R \eta_i+\gamma\eta_i^2)B_{2,i}\right]=0, \\
     \textstyle\sum_{i=1}^4\eta_i^{n-1}\big[\gamma \eta_i^{-2}+R \eta_i^{-1} + 2S-\gamma  +(R+ 2\gamma) \eta_i\big] B_{2,i}=0, \\
     \textstyle\sum_{i=1}^4\eta_i^{n}\big[\gamma \eta_i^{-2}+(R+2\gamma) \eta_i^{-1} + S-2\gamma-\i t \big] B_{2,i}=0.
 \end{array} \right.
\end{eqnarray}
We remark that, although the above system is underdetermined, $G_{h,j,2}$ is uniquely determined by \eqref{eGh}. As a matter of fact, \eqref{abc4} can be viewed as a system of five unknowns $B_{2,i}, i=1,2,3,4$ and $\sum_{i=1}^4\eta_i A_{2,i}$. As we just mentioned, a solution of \eqref{abc4} can be obtained from \eqref{eAB} with $m=2$, because of the following facts. The last three equations of \eqref{abc4} are the
same as the last three equations of \eqref{abc1} (with $m=2$). The
first equation of \eqref{abc4} can be obtained from the sum of the
first equation of \eqref{abc1} and the fourth equation of
\eqref{abc2} (with $m=2$). Similarly, the second equation of \eqref{abc4} by substracting the second equation of
\eqref{abc2} from the fifth equation of \eqref{abc1} (with $m=2$).

For $m=1$, the system \eqref{acc1} is reduced to the
system of four equations: $(V_1+V_2) B_1=z$.

In the following, we will solve \eqref{eAB}. First, assuming that the matrices used are all invertible, implying that their
determinants are not equal zero. Then, we will get $A_m=-V^{-1}V_2U_m^{-1} z$, $B_m=V^{-1}V_1 U_m^{-1} z$, $1<m\leq n$, and we can
also know $B_1=V^{-1} z$, where $V=V_1+V_2$.

\textbf{Step 2. Estimating $a_i$ and $b_i$.} In order to estimate $A_m$ and $B_m$, we prove in this step the following assertions:
\begin{align}
&|a_1|=|a_2|\leq t^2,\ \ a_3=a_4\eta_4^{-2n},\ \ |a_4|\geq 6|\eta_4|^{n},\label{aaa3} \\
&|b_1|>\frac{5}{3}t,\ \ |b_2|<\frac{t^2}{3},\ \ |b_3|<\frac{2}{3}|\eta_4|^{1-n},\ \ |b_4|<\frac{3}{2}|\eta_4|^{n},\label{aaa4}\\
&|a_1b_2-a_2b_1|=\abs{t^2(\eta_1-\eta_2)}\le 2t^2, \ \ \abs{a_3b_4-b_3a_4}\le 2t^2\abs{\eta_4}^{-n}\abs{a_4}.\label{aaa5}
\end{align}
where $\eta_4$ satisfies \eqref{bbb2}.

It follows from \eqref{ac} that
\begin{align*}
\abs{a_1}=\abs{2\cos t_h^{-}-2}\abs{\eta_1^n}\le t^2,\ \ \abs{a_2}=\abs{2\cos t_h^{-}-2}\abs{\eta_2^n}\le t^2.
\end{align*}
Using the identity $\eta_3=\eta_4^{-1}$ and \eqref{eai} we get
\begin{align*}
a_3=\eta_3^n(\eta_3+\eta_4-2)=\eta_4^{-2n}a_4.
\end{align*}
It follows from \eqref{bbb3} and \eqref{eai} that
\begin{align*}
\abs{a_4}=\abs{\eta_4^n(\eta_3+\eta_4-2)}=\abs{\eta_4^n}\abs{2\cos
t_h^{+}-2}\geq6\abs{\eta_4^n}.
\end{align*}
Therefore \eqref{aaa3} holds.

Next, we turn to prove \eqref{aaa4}. Noting that $0<\abs{\ga}< 1/6$,
from \eqref{ebi}, \eqref{bbb1}, and \eqref{e3.3} we have
\begin{align*}
\abs{b_2}&=\abs{1-\frac{t^2}{3}-\Big(1+\frac{t^2}{6}\Big)\eta_2^{-1}+\gamma(2\cos t_h^{-}-2)\eta_2^{-1}-\i t}\\
\nn&=\abs{1-\frac{t^2}{3}-\Big(1+\frac{t^2}{6}\Big)\eta_2^{-1}+\frac{1+\frac{t^2}{6}-\sqrt{\big(1+\frac{t^2}{6}\big)^2+4\gamma
t^2}}{2}\eta_2^{-1}-\i t}\\
\nn&\leq\abs{1-\frac{t^2}{3}-\frac{1+\frac{t^2}{6}+\sqrt{\big(1+\frac{t^2}{6}\big)^2+4\gamma
t^2}}{2}\cos t_h^{-}} \\
\nn &\ \ \ \ \ \ \ \ \ +\abs{\frac{1+\frac{t^2}{6}+\sqrt{\big(1+\frac{t^2}{6}\big)^2+4\gamma
t^2}}{2}\sin t_h^{-}-t} :=\rm{(I)}+(II),
\end{align*}
where
\begin{align*}
\textnormal{(I)}=&\abs{1-\frac{t^2}{3}-\frac{1+\frac{t^2}{6}+\sqrt{\big(1+\frac{t^2}{6}\big)^2+4\gamma
t^2}}{2}\bigg(1-\frac{t^2}{1+\frac{t^2}{6}+\sqrt{\big(1+\frac{t^2}{6}\big)^2+4\gamma
t^2}}\bigg)}\\
=&\abs{\frac{\sqrt{\big(1+\frac{t^2}{6}\big)^2+4\gamma
t^2}-1-\frac{t^2}{6}}{2}}\le\frac{t^2}{6},
\end{align*}
\begin{align*}
\rm{(II)}=&\abs{\frac{1+\frac{t^2}{6}+\sqrt{\big(1+\frac{t^2}{6}\big)^2+4\gamma
t^2}}{2}\sqrt{1-(\cos t_h^{-})^2}-t}\\
=&\abs{\frac{\sqrt{2t^2\Big(1-\frac{t^2}{3}+\sqrt{\big(1+\frac{t^2}{6}\big)^2+4\gamma
t^2}\Big)}-2t}{2}}\\
=&\frac{\abs{-\frac{1}{3}-\frac{t^2}{12}+4\gamma
}t^3}{\Big(\sqrt{2\Big(1-\frac{t^2}{3}+\sqrt{\big(1+\frac{t^2}{6}\big)^2+4\gamma
t^2}\Big)}+2\Big)\Big(1+\frac{t^2}{3}+\sqrt{\big(1+\frac{t^2}{6}\big)^2+4\gamma
t^2}\Big)}< \frac{t^3}{6},
\end{align*}
we therefore arrive at
\begin{equation}
\abs{b_2}\leq \textnormal{(I)+(II)}<\frac{t^2}{3}.
\end{equation}
Noting that $\bar\eta_2=\eta_1$, it is clear that
\begin{align*}
b_1=\eta_1^n\Big(1-\frac{t^2}{3}-\Big(1+\frac{t^2}{6}\Big){\eta_1^{-1}}+\gamma(1-{\eta_1^{-1}})^2-\i
t\Big)=\bar{b}_2-2\i t\eta_1^n.
\end{align*}
Obviously, $|b_1|\ge 2t-\abs{b_2}>\frac{5}{3}t$. From \eqref{bbb2},
\begin{align*}
\abs{b_3}&=\abs{\eta_3^n\Big(1-\frac{t^2}{3}-\Big(1+\frac{t^2}{6}\Big)\eta_3^{-1}+\gamma(1-\eta_3^{-1})^2-\i t\Big)}\\
\nn&=\abs{\eta_3^n\Big(1-\frac{t^2}{3}-\Big(1+\frac{t^2}{6}\Big)\eta_3^{-1}+\gamma(2\cos t_h^{+}-2)\eta_3^{-1}-\i t\Big)}\\
\nn &=\abs{\eta_3^n}\abs{1-\frac{t^2}{3}+\frac{2\gamma
t^2}{1+\frac{t^2}{6}+\sqrt{\big(1+\frac{t^2}{6}\big)^2+4\gamma
t^2}}\eta_3^{-1}-\i t}\\
&\le\abs{\eta_3^{n-1}}\bigg(\frac13\abs{1-\frac{t^2}{3}-\i t}+\frac{
t^2}{6}\bigg)
<\frac{2}{3}\abs{\eta_3}^{n-1}=\frac{2}{3}\abs{\eta_4}^{1-n}.
\end{align*}
Similarly,
\begin{equation*}
\abs{b_4}=\abs{\eta_4^n}\abs{1-\frac{t^2}{3}+\frac{2\gamma
t^2}{1+\frac{t^2}{6}+\sqrt{\big(1+\frac{t^2}{6}\big)^2+4\gamma
t^2}}\eta_4^{-1}-\i t}<\frac{3}{2}\abs{\eta_4^n}.
\end{equation*}
This completes the proof of \eqref{aaa4}.

It remains to prove \eqref{aaa5}. We derive from  \eqref{eai}--\eqref{ebi}, \eqref{bbb1}, and \eqref{e3.3} that
\begin{align*}
|a_1b_2-a_2b_1|&=\Big|(\eta_1+\eta_2-2)\Big(1-\frac{t^2}{3}-\Big(1+\frac{t^2}{6}\Big)\eta_1+\gamma(1-\eta_1)^2-\i t\Big)\\
& -(\eta_1+\eta_2-2)\Big(1-\frac{t^2}{3}-\Big(1+\frac{t^2}{6}\Big)\eta_2+\gamma(1-\eta_2)^2-\i t\Big)\Big|\\
& =\abs{(\eta_1+\eta_2-2)\Big(\ga (\eta_1+\eta_2-2)(\eta_1-\eta_2)-\Big(1+\frac{t^2}{6}\Big)(\eta_1-\eta_2)\Big)}\\
&=\abs{t^2(\eta_1-\eta_2)}\le 2t^2.
\end{align*}
Similarly,
\begin{align*}
a_3b_4-b_3a_4&=t^2(\eta_3-\eta_4)=t^2a_4\eta_4^{-n}\frac{\eta_3-\eta_4}{\eta_3+\eta_4-2}=t^2a_4\eta_4^{-n}\frac{1-\eta_4^2}{1+\eta_4^2-2\eta_4},
\end{align*}
hence, again from \eqref{bbb2},
\begin{equation*}
\abs{a_3b_4-b_3a_4}=t^2\abs{a_4\eta_4^{-n}\frac{1-\eta_4}{1+\eta_4}}\le t^2\abs{a_4\eta_4^{-n}}\frac{\abs{\eta_4}+1}{\abs{\eta_4}-1}\le 2t^2\abs{\eta_4}^{-n}\abs{a_4}.
\end{equation*}
This completes the proof of \eqref{aaa5}.

\textbf{Step 3. Estimating $A_m$ and $B_m$.} Since $t=kh\leq 1$ and $k\geq 1$, from \eqref{bbb2}, we have
\begin{align}\label{aaa2}
 |\eta_4|^{-n}<\Big(\frac{1}{3}\Big)^{\frac{1}{h}}\leq
\Big(\frac{1}{3}\Big)^{\frac{1}{t}}\le \frac{1}{3}t.
\end{align}

Next we estimate $\dfrac1{\det V}$. By some simple calculation, we have
\begin{equation}\label{edetV}
\det V= [(\eta_3+\eta_4)-(\eta_1+\eta_2)][ (a_2-a_1)(b_4-b_3)-
(b_2-b_1)(a_4-a_3)]
\end{equation}
where $a_i$ and $b_i$ is defined by \eqref{eai} and \eqref{ebi} respectively.
We analyze and estimate each term of $\det V$. From \eqref{aaa4}, it is clear that
 $\big|\dfrac{b_2}{b_1}\big|<\dfrac{t}{5}$. Hence,
\begin{equation}\label{aa3}
b_1-b_2=b_1(1+\theta_1 t) \ \  \text{where $\theta_1$ is a general function satisfying }
|\theta_1|<\frac{1}{5}.
\end{equation}
It follows from \eqref{aaa2}, \eqref{aaa3}, and \eqref{aa3} that
\begin{equation}\label{bb3}
(b_1-b_2)(a_4-a_3)=b_1 a_4 (1+\theta_2 t)  \ \  \text{where $\theta_2$ is a general function and }
|\theta_2|<\frac{1}{3}.
\end{equation}
From \eqref{aaa3}--\eqref{aaa4} and \eqref{aaa2}, we have
\begin{equation*}
|(a_2-a_1)(b_4-b_3)|\leq \frac{2}{3}t^2 |a_4|\leq \frac{2}{5}t|b_1a_4|.
\end{equation*}
It follows from \eqref{edetV}, \eqref{bb3}, and the above inequality that
\begin{equation*}
\det{V}=b_1 a_4 [(\eta_3+\eta_4)-(\eta_1+\eta_2)](1+ \theta_3 t),
\end{equation*}
where $\theta_3$ is a general function and $\abs{\theta_3}<\dfrac{11}{15}$. Therefore from Lemma~ \ref{LA1},
\begin{align}\label{Vinv}
\frac{1}{\det V} &= \frac{1+ \widetilde{C}t} {b_1 a_4
\sigma},
\end{align}
where $\sigma:=\eta_3+\eta_4-(\eta_1+\eta_2)$. Note from \eqref{e3.3} and \eqref{bbb1} that
\begin{align}\label{esigma}
\frac{1}{\sigma}=\frac{\gamma }{\sqrt{\left(1+\frac{t^2}{6}\right)^2+4\gamma t^2}}= \gamma(1+ \widetilde{C}t^2).
\end{align}

In order to estimate $B_1$, we consider the first column of $V^*$, the adjugate of $V$. From \eqref{aaa2} and \eqref{aaa3}--\eqref{aaa5}, by some calculations,  we have
\begin{align*}
V^*(:,1)=\left(\begin{array}{c}
a_3b_4-b_3a_4+b_2a_4-a_2b_4+a_2b_3-b_2a_3\\
-b_1a_4+a_1b_4+ b_3a_4-a_3b_4+b_1a_3-a_1b_3\\
b_1a_4-a_1b_4+a_2b_4-b_2a_4+a_1b_2-a_2b_1\\
a_3b_2-a_2b_3+a_1b_3-b_1a_3+a_2b_1-a_1b_2 \end{array}\right)
=\left(\begin{array}{c}
b_1 a_4\widetilde{C}t  \\
-b_1 a_4(1+\widetilde{C}t)  \\
 b_1 a_4 (1+\widetilde{C}t)  \\
 \eta_4^{-n} b_1 a_4 \widetilde{C}t  \end{array}\right),
\end{align*}
hence, from \eqref{Vinv} and \eqref{esigma},
\begin{align}\label{aaaa1}
B_1=V^{-1}z=\frac{1}{\det V}V^*z=\frac{1}{\det V}\Big(-\frac{1}{\ga} \Big)V^*(:,1)=\left(\begin{array}{c}\widetilde{C}t\\
1+\widetilde{C}t\\-1+\widetilde{C}t\\ \eta_4^{-n} \widetilde{C}t
\end{array}\right).
\end{align}

We turn to estimate $A_m$ and $B_m$ for $m>1$. It follows from the definitions of $U_m$ and $z$ that,
\begin{equation}\label{aa1}
U_m^{-1}z=-\frac{1}{\ga}\left( \begin
{array}{c}  \dfrac{\eta_1^{2-m}\eta_2\eta_3\eta_4}{(\eta_3-\eta_1)(\eta_4-\eta_1)(\eta_2-\eta_1)}\\
\noalign{\medskip}  \dfrac{\eta_2^{2-m}\eta_1\eta_3\eta_4}{(\eta_3-\eta_2)(\eta_4-\eta_2)(\eta_1-\eta_2)} \\
\noalign{\medskip} \dfrac{\eta_3^{2-m}\eta_1\eta_2\eta_4}{(\eta_1-\eta_3)(\eta_2-\eta_3)(\eta_4-\eta_3)}\\
\noalign{\medskip} \dfrac{\eta_4^{2-m}\eta_1\eta_2\eta_3}{(\eta_1-\eta_4)(\eta_2-\eta_4)(\eta_3-\eta_4)}\end{array} \right)=(1+ \widetilde{C}t^2)\left( \begin
{array}{c}\dfrac{\eta_2^{m}}{\eta_2-\eta_1}\\
\noalign{\medskip} \dfrac{\eta_1^{m}}{\eta_1-\eta_2}\\
\noalign{\medskip} \dfrac{\eta_4^{m}}{\eta_3-\eta_4}\\
\noalign{\medskip} \dfrac{\eta_3^{m}}{\eta_4-\eta_3}\end{array} \right),
\end{equation}
where we have used \eqref{bbb1} and \eqref{esigma} to derive the last equality.

Next we estimate $V^*V_1$. Clearly, $V_1(:,2)=V_1(:,1)$, $V_1(:,4)=V_1(:,3)$, and so is $V^*V_1$. It follows from \eqref{aaa3}--\eqref{aaa5} and \eqref{aaa2} that,
\begin{align}\label{aa2}
&V^*V_1(:,[1,3])=V^*V_1(:,[2,4])\\
\nn&=\sigma \left(\begin{array}{cc}
 a_2b_4-a_2b_3-b_2a_4+b_2a_3  & a_3b_4-b_3a_4 \\
 a_1b_3-a_1b_4+b_1a_4-b_1a_3  & b_3a_4-a_3b_4 \\
 a_2b_1-a_1b_2  & a_2b_4-a_1b_4-a_4b_2+a_4b_1 \\
 a_1b_2-b_1a_2  & a_3b_2-a_3b_1-b_3a_2+b_3a_1
 \end{array}\right)\\
\nn&=\sigma b_1 a_4 \left(\begin{array}{cccc}
\widetilde{C}t\ & \eta_4^{-n}\widetilde{C} t  \\
1+\widetilde{C}t & \eta_4^{-n}\widetilde{C} t \\
\eta_4^{-n}  (\eta_1-\eta_2)\widetilde{C}t & 1+\widetilde{C}t \\
\eta_4^{-n} (\eta_1-\eta_2) \widetilde{C}t &-\eta_4^{-2n}(1-\eta_4\widetilde{C}t)
\end{array}\right).
\end{align}
From \eqref{Vinv},  \eqref{aa1}, \eqref{aa2}, \eqref{bbb2}, and $\abs{\eta_1}=\abs{\eta_2}=1$, we have
\begin{align}\label{eBm}
B_m&=V^{-1}V_1U_m^{-1}z=\frac{1}{\det V}V^*V_1U_m^{-1}z\\
&=(1+\widetilde{C}t)\left( \begin{array}{c}
\dfrac{\eta_2^m-\eta_1^m}{\eta_2-\eta_1}\,\widetilde{C}t+\widetilde{C}t\\
\dfrac{\eta_2^m-\eta_1^m}{\eta_2-\eta_1}(1+\widetilde{C}t)+\widetilde{C}t\\
\eta_4^{-n}\widetilde{C}t+\eta_4^{m-1}\widetilde{C}\\
\eta_4^{-n}\widetilde{C}t+\eta_4^{-2n-1+m}\widetilde{C}\end{array} \right)
=\left( \begin{array}{c}
\dfrac{\widetilde{C}t}{\sin t_h^-}\\
\dfrac{\sin(m t_h^-)+\widetilde{C}t}{\sin t_h^-}\\
\eta_4^{m-1}\widetilde{C}\nn\\
\eta_4^{-n}\widetilde{C}t+\eta_4^{-2n-1+m}\widetilde{C}\end{array} \right).
\end{align}
Similarly, again from \eqref{aaa2}, \eqref{Vinv} and \eqref{aaa3}--\eqref{aaa5},
\begin{align}\label{aa4}
&(V^{*}V_2)([1,3],:)=-(V^{*}V_2)([2,4],:)\\
=&\sigma b_1a_4 \left(\begin{array}{cccc}
1+\widetilde{C}t & \widetilde{C}t & \eta_4^{-n}\widetilde{C} t & \eta_4^{-n}\widetilde{C} t \\
\eta_4^{-n}  (\eta_1-\eta_2)\widetilde{C}t & \eta_4^{-n} (\eta_1-\eta_2)\widetilde{C}t & \eta_4^{-2n}(\eta_4\widetilde{C}t-1) & -1+\widetilde{C} t\\
\end{array}\right).\nn
\end{align}
It follows from \eqref{Vinv}, \eqref{aa1} and \eqref{aa4} that,
\begin{align}\label{eAm}
A_m([1,3])=-A_m([2,4])&=-\frac{1}{\det V}(V^{*}V_2)([1,3],:)U_m^{-1}z\\
&=\left(\begin{array}{c}
\dfrac{\eta_2^{m}}{\eta_1-\eta_2}+\dfrac{\widetilde{C}t}{\sin t_h^-}\\
\eta_4^{-m}\widetilde{C}t+\eta_4^{-m-1}\widetilde{C}\\
\end{array}\right).\nn
\end{align}
\textbf{Step 4. Finishing up.} It is time to consider $H_{h,j,m}$.
Let $w_{1}^T=[\eta_1, \eta_2, \eta_3, \eta_4]$, $w_{j}^T=[(\eta_1-1)\eta_1^{j-1}, (\eta_2-1)\eta_2^{j-1},
(\eta_3-1)\eta_3^{j-1}, (\eta_4-1)\eta_4^{j-1}]$ for $j>1$. From \eqref{eGh}, \eqref{eHh}, \eqref{eBm}, and \eqref{eAm}, we have,
for $m=1$,
\begin{align*}
H_{h,1,1}&=G_{h,1,1}=w_{1}^T B_1=e^{\i t_h^-}+\widetilde{C}\eta_4^{-1}+\widetilde{C}t=\widetilde{C}=\i \sin (t_h^-)e^{\i t_h^-}+\widetilde{C}t+\widetilde{C} ,\\
H_{h,j,1}&=G_{h,j,1}-G_{h,j-1,1}=w_{j}^T B_1=\i \sin (t_h^-) e^{\i j t_h^-}+\widetilde{C}t+\widetilde{C}\eta_4^{1-j},\ \ j>1,
\end{align*}
and for $m>1$,
\begin{align*}
H_{h,1,m}&=G_{h,1,m}=w_{1}^T A_m=\cos(t_h^-)e^{\i mt_h^-}+\widetilde{C}t+\widetilde{C}\eta_4^{1-m},\\
H_{h,j,m}&=G_{h,j,m}-G_{h,j-1,m}=w_{j}^T A_m=\cos(j t_h^-)e^{\i m t_h^-}+\widetilde{C}t+\widetilde{C}\eta_4^{j-m},\ \ \ 1<j<m,\\
H_{h,j,m}&=G_{h,j,m}-G_{h,j-1,m}=w_{j}^T B_m=\i \sin(m t_h^-)e^{\i j t_h^-}+\widetilde{C}t+\widetilde{C}\eta_4^{m-j},\ \ j>m,\\
H_{h,m,m}&=G_{h,m,m}-G_{h,m-1,m}=\sum_{i=1}^4 B_{m,i}\eta_i^m-\sum_{i=1}^4 A_{m,i}\eta_i^{m-1}\\
&=w_{m}^T B_m+\sum_{i=1}^4 (B_{m,i}-A_{m,i})\eta_i^{m-1}=w_{m}^T B_m\\
&=\i \sin(m t_h^-)e^{\i m t_h^-}+\widetilde{C}t+\widetilde{C},
\end{align*}
where we have used $\sum_{i=1}^4 (B_{m,i}-A_{m,i})\eta_i^{m-1}=0$ (cf. the sixth equation in \eqref{abc2}). This completes the proof of the lemma.
\end{proof}

From Lemma~\ref{lHh2} and \eqref{euh'1}, we have
\begin{align}\label{bb4}
    u_h'(x)=&\sum_{m=1}^{n} H_{h,j,m}(f, \phi_m)=\sum_{m=1}^{j}\i \sin{(m t_{h}^{-})}e^{\i j t_{h}^{-}}(f,\phi_m)\\
  \nn & +\sum\limits_{m=j+1}^{n}\cos{(j t_{h}^{-})}e^{\i m t_{h}^{-}}(f,\phi_m)+ t\sum\limits_{m=1}^{n}\widetilde{C}\,(f,\phi_m)\\
   \nn & +\sum\limits_{m=1}^{n}{\eta_4}^{-\abs{m-j}}\widetilde{C}\,(f,\phi_m),\quad \forall x\in [x_{j-1},x_j], \ 1\leq j\leq n.
\end{align}
Comparing with the continuous case \eqref{eu'H} we see that the
first two contributions in the right hand side of \eqref{bb4}
consists of the discrete travelling wave, whereas the last two
perturbed terms will be shown to be of the same order as the
interpolation error.

\section{Stability and Pre-asymptotic error estimates for the CIP-FEM}
In this section, we
consider the stability and error estimates of the CIP-FEM solution in the discrete semi-norm $\norm{\cdot}_{1,h}$ for real penalty parameters.

\begin{mythm}\label{tsta}
Under the conditions of Lemma~\ref{lHh2}, the CIP method
\eqref{eipdg} attains a unique solution $u_h$ that satisfies the
stability estimate
\begin{equation}
\norm{u_h}_{1,h}\lesssim \norm{f}.
\end{equation}
\end{mythm}
\begin{proof}
Let us estimate each term in the definition of $\norm{\cdot}_{1,h}$(cf. \eqref{enorm2}). First, from \eqref{bb4}, it is clear that
\begin{equation}
\nn \abs{u_h'(x)}\lesssim \sum_{m=1}^{n} \abs{(f, \phi_m)}\lesssim \norm{f} ,\quad  \forall x\in
[x_{j-1},x_j] ,\quad  j=1,\cdots,n ,
\end{equation} and hence,
\begin{equation}
\norm{u_h'}\lesssim \norm{f}.
\end{equation}
Secondly,
\begin{equation}
\nn\abs{[u_h']_j}=\abs{u_h'(x_j+)-u_h'(x_j-)}\leq
\abs{u_h'(x_j+)}+\abs{u_h'(x_j-)}\lesssim \norm{f},
\end{equation}
which impies
\begin{equation}
\nn\sum_{j=1}^{n-1} |\gamma| h\abs{[u_h']_j}^2\lesssim \norm{f}^2.
\end{equation}Therefore,
\begin{equation}
\nn\norm{u_h}_{1,h}={(\abs{u_h}_{1,h}^2+\sum_{j=1}^{n-1} |\gamma|
h\abs{[u_h']_j}^2)}^\frac{1}{2}\lesssim \norm{f}.
\end{equation}
This completes the proof of Theorem~\ref{tsta}.
\end{proof}
\begin{myrem} This stability estimate for the CIP-FEM (as well as FEM) is of the same order as that of the continuous problem (cf. \eqref{e2.2b}). Note that the estimate holds for real penalty parameters in $[-\frac16, \frac16]$ under the condition $kh\le 1$ in current one-dimensional setting. The same result has been proved for the one-dimensional FEM in \cite{F.I.I}. For stability estimates of the CIP-FEM for higher-dimensional problems, we refer to  \cite{H.Wu} which, particularly, gives estimates for imaginary penalty parameters under the condition  $k^3h^2\lesssim 1$.
\end{myrem}

\begin{mythm}\label{terr3}
Under the conditions of Lemma~\ref{lHh2},
\begin{equation}\label{error1}
    \norm{u-u_h}_{1,h} \lesssim (k h+ \abs{k_h^- -k} )\norm{f}\lesssim (k h+ k^3 h^2 )\norm{f}.
\end{equation}
    If, furthermore, $\gamma=-\frac{1}{12}$,\quad then
\begin{equation}
    \norm{u-u_h}_{1,h} \lesssim (k h+ k^5 h^4 )\norm{f}.
\end{equation}
    If, furthermore, $\abs{\gamma-\gamma_o}\lesssim \frac{1}{k^2 h}$,\quad then
\begin{equation}\label{error3}
    \norm{u-u_h}_{1,h} \lesssim k h \norm{f}.
\end{equation}
    Here $\ga_o$ is defined in Lemma~ \ref{lkh3}.
\end{mythm}

\begin{proof}
Suppose $n \lesssim k^2$, that is, $k^2h\gtrsim 1$,  otherwise, \eqref{error3} is proved by using the Schatz argument \cite{sch74}. To estimate the last perturbed term in \eqref{bb4}, define $q_0$ to be the largest integer less than or equal to $-\ln t/\ln3$. From \eqref{bbb2}, it is clear that
\begin{equation}\label{aaa1}
 \abs{\eta_4}^{-q}<3^{-q}<
t \text{ for } q>q_0 \text{ and } q_0 \lesssim \ln k\lesssim k.
\end{equation}
 Define
\begin{align*}
 \phi_0:=\left\{ \begin{array}{ll}
\frac{x_1-x}{h}, &\quad 0 \leq x \leq x_1, \\
0,&\quad  x>x_1.
\end{array} \right.
\end{align*}
Denote by $x_j=0$ for $j<0$ and $x_j=1$ for $j>n$. We make use of the formulation
of $u'(x)$ in \eqref{eu'H} and the characterization of $u_h'(x)$ in
\eqref{bb4} to obtain: For $x\in K_j$, $j=1,2,\cdots,n$,
\begin{align*}
    &\abs{u'(x)-u_h'(x)}=\Big|\int_0^1H(x,s)f(s)\sum_{m=0}^n\phi_m(s)\ds-u_h'(x)\Big|\\
   &\lesssim \Big|\int_0^1H(x,s)f(s)\phi_0(s)\ds\Big|+\sum_{m=1}^{j}\int_{x_{m-1}}^{x_{m+1}}\abs{\big(H(x,s)-\i\sin(m t_h^-) e^{\i j t_h^-} \big)f\phi_m}\ds\\
    &\quad +\sum_{m=j+1}^n\int_{x_{m-1}}^{x_{m+1}}\abs{\big(H(x,s)-\cos(j t_h^-)  e^{\i m t_h^-} \big)f\phi_m}\ds
          +t\norm{f}\\
    &\quad +\sum_{m=1}^{n}\int_{x_{m-1}}^{x_{m+1}}\abs{\eta_4}^{-\abs{j-m}}\abs{f}\ds\\
    &\lesssim \int_{0}^{x_1}|f|\ds+\sum_{m=1}^{j-2}\int_{x_{m-1}}^{x_{m+1}}\abs{\big(\i\sin{ks}\,e^{\i kx}-\i\sin(m t_h^-) e^{\i j t_h^-} \big)f\phi_m}\ds  \\
    &\quad+\int_{x_{j-2}}^{x_{j+1}}|f|\ds+\sum_{m=j+1}^n\int_{x_{m-1}}^{x_{m+1}}\abs{\big(\cos{kx}\,e^{\i ks}-\cos(j t_h^-)  e^{\i m t_h^-} \big)f\phi_m}\ds\\
    &\quad + t\norm{f}+\int_{x_{j_1}}^{x_{j_2}}\abs{f}\ds
    \db\\
     &\lesssim \sum_{m=1}^{n}\big((m+j)\abs{t_h^--t}+ t\big)(\abs{ f},\phi_m)\\
     &\quad+h^{\frac12}\norml{f}{[x_0, x_1]\cup [x_{j-2}, x_{j+1}]}+(q_0h)^{\frac12}\norml{f}{[x_{j_1}, x_{j_2}]}+ t\norm{f},
\end{align*}
where $j_1=\max \{j-q_0-1,0\}$, $j_2=\min \{j+q_0+1,n\}$ and we have used the Lagrange Mean Value Theorem to derive the last inequality.
Noting that $(m+j)\abs{t_h^- -t}=(m+j)h\abs{k_h^--k}\le 2\abs{k_h^- -k}$, the above inequality yields
\begin{align*}
    \abs{u'(x)-u_h'(x)}\lesssim &(t+\abs{k_h^--k})\norm{f}+h^{\frac12}\norml{f}{[x_0, x_1]}+h^{\frac12}\norml{f}{[x_{j-2}, x_{j+1}]}\\
    &+(q_0h)^{\frac12}\norml{f}{[x_{j_1}, x_{j_2}]},\qquad
    \forall x\in K_j, \  j=1,\cdots,n.
\end{align*}
As direct consequences of the above inequality, we have
\begin{align}\label{eterr2}
    \norm{(u-u_h)'}_{L^2(\Om)}^2&\lesssim
    \big(t+\abs{k_h^- -k} \big)^2\norm{f}^2+q_0^2h^2\norm{f}^2+h\norm{f}^2\\
    &\lesssim \big(t+\abs{k_h^- -k} \big)^2\norm{f}^2,\nn
\end{align}
where we have used $q_0h\lesssim t$ (cf. \eqref{aaa1}) and $h\lesssim t^2$ (since $k^2h\gtrsim 1$) to derive the last inequality.
\begin{align*}
\abs{[(u-u_h)']_j}&=\abs{(u'(x_j+)-u_h'(x_j+))-(u'(x_j-)-u_h'(x_j-))}\\
&\leq \abs{u'(x_j+)-u_h'(x_j+)}+\abs{u'(x_j-)-u_h'(x_j-)}\\
&\lesssim
(t+\abs{k_h^- -k})\norm{f}+h^{\frac12}\norml{f}{[x_0, x_1]}+h^{\frac12}\norml{f}{[x_{j-3}, x_{j+2}]}\\
    &\quad+(q_0h)^{\frac12}\norml{f}{[x_{(j-1)_1}, x_{(j+1)_2}]}.
\end{align*}
Since $|\ga|\leq 1/6$,
\begin{align}
\sum_{j=1}^{n-1} |\gamma| h\abs{[(u-u_h)']_j}^2&\lesssim
(t+\abs{k_h^- -k})^2\norm{f}^2+q_0^2h^2\norm{f}^2+h\norm{f}^2\\
    &\lesssim \big(t+\abs{k_h^- -k} \big)^2\norm{f}^2,\nn
\end{align}
which together with \eqref{eterr2} implies \eqref{error1}.
By using Lemma~\ref{lkh3}, we can complete the proof of the theorem.
\end{proof}
\begin{myrem} (a) This theorem shows that the pollution error in $H^1$-norm is controlled by the phase difference $k-k_h^-$.
 Ihlenburg and Babu\v{s}ka \cite{F.I.I,F.I.I.B} obtained the same result for the FEM in the one dimensional case. Recently, the authors \cite{H.Wu, Z.W} showed  for the CIP-FEM and FEM in higher dimensions that the pollution errors in $H^1$-norm are of the same order as the phase difference obtained by dispersion analyses.

(b) The pollution effect of the CIP-FEM in one dimension can be eliminated by chosen appropriately the penalty parameters (cf. \eqref{error3}). It is well-known that, the pollution effect exists in the FEM while in one dimension, it can be eliminated by a
suitable modification of the discrete system but using the same stencil (cf. \cite{bs00}). Note that the stencil of the CIP-FEM ($\ga\neq 0$) is different from that of the FEM. We refer to \cite{dpg11} for similar results on discontinuous Petrov-Galerkin methods.

\end{myrem}
\section{Numerical Evaluation} Throughout this section, we
consider the BVP with constant right hand side $f(x)\equiv -1$.

\subsection{The discrete wavenumber}
Unlike the best approximation, the CIP-FEM solution is, in general,
not in  phase with the exact solution. Numerical tests show that the
discrete solution has a phase delay with respect to the exact
solution when $-\frac{1}{6}\leq \ga<\ga_o$ and has a phase lead with
respect to the exact solution when $\ga_o<\ga \leq \frac{1}{6}$
which is similar to the FEM solution~\cite{F.I.I}. Hence we can
choose an appropriate value of the stabilization parameter to
eliminate the phase error. ``Optimal" values of $\ga$ are those in a
neighbourhood of $\ga_o$. This is shown in Figure~\ref{nophaseerr},
where the real and the imaginary parts of both solutions are plotted
for $k=10$, $kh=1$. There is no phase error for the CIP-FEM solution.
\begin{figure}[!htb]
\begin{center}
\includegraphics[scale=0.6]{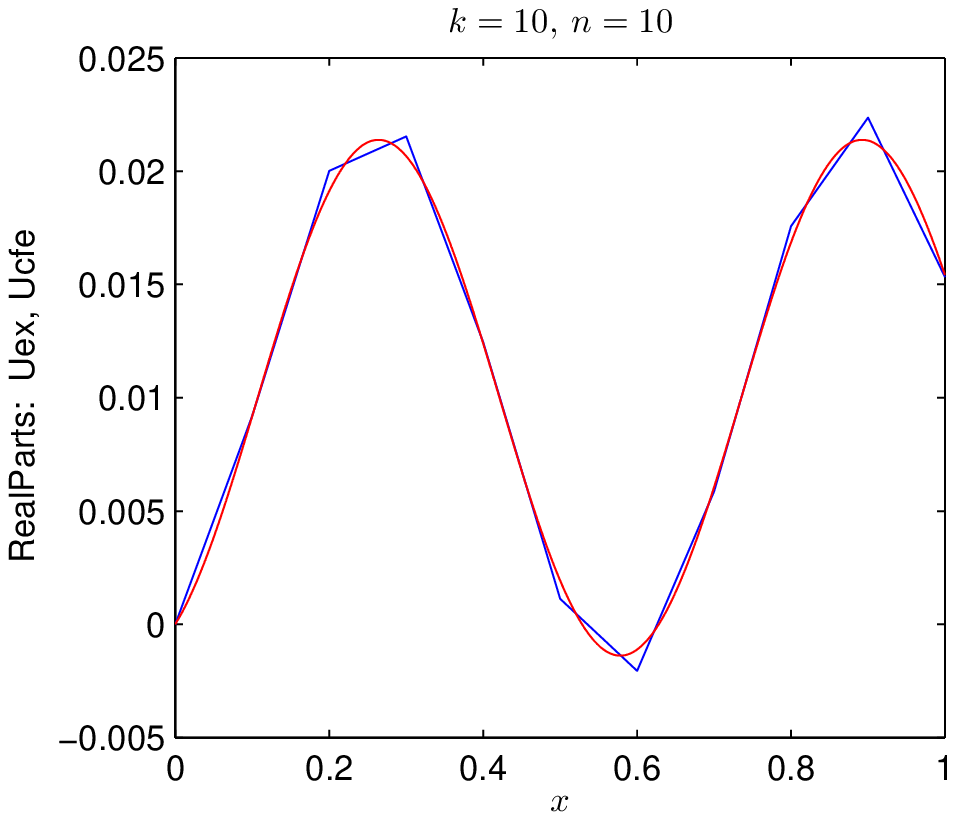}
\includegraphics[scale=0.6]{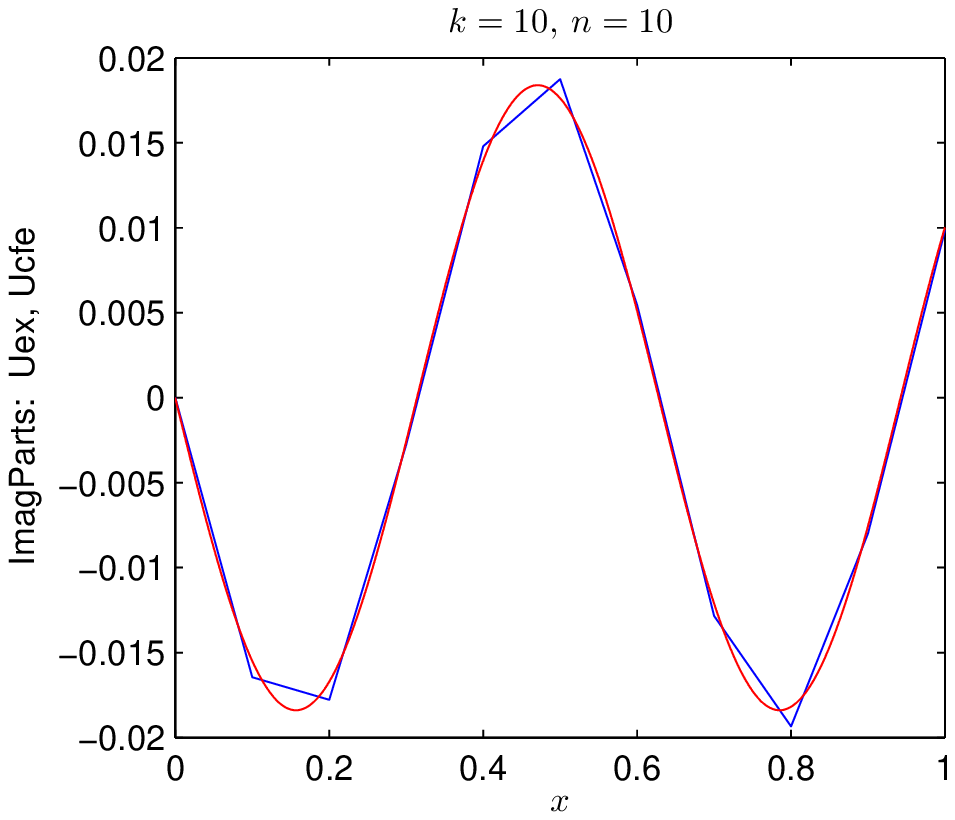}
\end{center}
\caption{\small{No phase error of the CIP-FEM solution with $\ga=\ga_o$ for
$k=10,n=10.$}}\label{nophaseerr}
\end{figure}

On a uniform mesh, the numerical dispersion relation of CIP method is
\begin{equation}\label{aa6}
\cos t_h^{-}(\ga)=\frac{4\gamma+1+\frac{t^2}{6}-
\sqrt{\left(1+\frac{t^2}{6}\right)^2+4\gamma t^2} }{4 \gamma },
\end{equation}
where $t=kh$. For fixed $\ga$, the right-hand is a  function of $t$, and is
used for computation of the discrete wavenumber that governs the periodicity of the CIP-FEM solution. In Figure~\ref{penparam}, the functions $y_1=\cos t=\cos t^-_h(\ga_o)$, $y_2=\cos t^-_h(-1/12)$, $y_3=\cos t^-_h(0)$ and $|y_4|=1$ are plotted. At $t_c=\sqrt{48 \ga +12}$, the functions $y_i$ ($i=2,3$) reach absolute value 1; the numerical solution switches from the propagating case to the decaying case. The value $t_c$ corresponds to a \textit{cutoff frequency} for the numerical solution \cite{L.L.T}.

For fixed $k$, the convergence $k_h^-(:=t_h^-/h) \rightarrow k$ is
visualized by $\cos t_h^- \rightarrow \cos t=\cos kh$ as $h\rightarrow
0$. The curves begin to deviate significantly at about $kh=t=1$.
\begin{figure}[!htb]
\begin{center}
\includegraphics[scale=0.6]{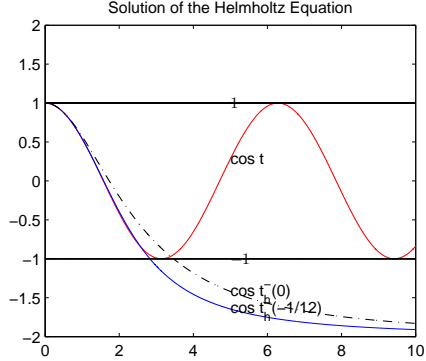}
\end{center}
\caption{\small{Convergence of discrete to exact wavenumber via comparison of $\cos t_h^{-}(\ga)$ for $\ga=\ga_o, -1/12, 0$ to $\cos t$}. The cutoff frequency $t_c=\sqrt{8}$ for $\ga=-1/12$, $t_c=\sqrt{12}$ for $\ga=0$.}\label{penparam}
\end{figure}
\subsection{Error of the best approximation and CIP-FEM solution} Consider in Figure~\ref{CIPopt} log-log-plots of the relative error $e_{ba}:= |u - u_I|_1/|u|_1$ of the best approximation and the relative error $e_c:=|u-u_h|_1/|u|_1$ by choosing $\ga=\ga_o$ for different $k$.
Note that the errors first stay at $100\%$ on coarse mesh, then
start to decrease at a certain meshsize, and then decrease with
constant slope of $-1$ (in log-log scale). This illustrates that the
CIP-FEM solution is convergent to the best approximation and there is no
pollution error for the solution. We are interested in the critical
number of DOF where the relative error begins to decrease (see for
instance \cite{F.I.I}). We can see from Figure \ref{CIPopt} that the
critical numbers of DOF for both the best approximation  and  the
CIP-FEM solution with $\ga=\ga_o$ are about $N=[\frac{k}{\pi}]$.
\begin{figure}[!htb]
\begin{center}
\includegraphics[scale=0.6]{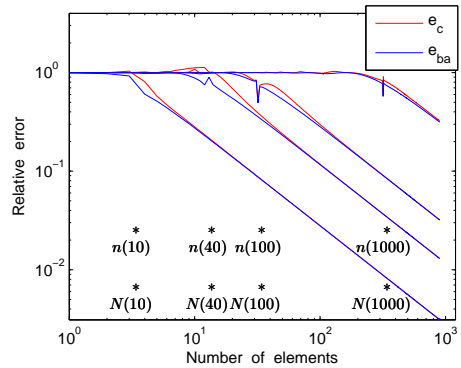}
\end{center}
\caption{\small{The relative error of the best approximation and CIP-FEM
solution with $\ga=\ga_o$ in $H^1$-seminorm and predicted critical
numbers of DOF for $k = 10$, $k = 40$, $k = 100$ and
$k=1000$.}}\label{CIPopt}
\end{figure}

For general $\ga$, the critical number of DOF $N_c$ can be predicted
using the  methods of \cite{F.I.I}:
$$ |k_h^{-}-k|\leq \frac{\pi}{3}\approx 1.$$
If $\ga$ does not depend on $t$, $N_c$ follows from the Taylor expansion equation \eqref{aa6}:
$$ N_c=\Big(\frac{|12\ga+1|}{24}k^3\Big)^{\frac{1}{2}}\ \ \ (\ga\neq -\frac{1}{12}),\ \ \ \ \ N_c=\Big(\frac{k^5}{720}\Big)^{\frac{1}{4}}\ \ \ (\ga=-\frac{1}{12}).$$
The formula of the critical number of DOF for CIP-FEM solution is
similar to FEM solution when $\ga\neq -1/12$, we consider the
$\ga=-1/12$ case in Figure \ref{relerrCIP}. It shows that the
predicted critical number of DOF is very good, especially for large
$k$.
\begin{figure}[!htb]
\begin{center}
\includegraphics[scale=0.6]{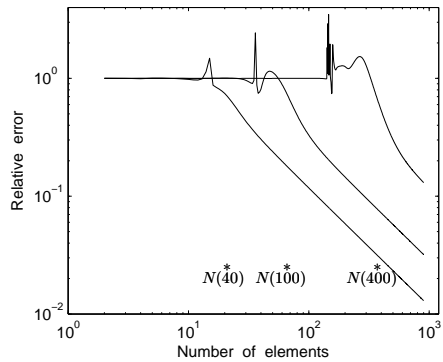}
\end{center}
\caption{\small{The relative error of the CIP-FEM solution with
$\ga=-1/12$ in $H^1$-seminorm and predicted critical numbers of DOF
for $k=40,100,400$.}}\label{relerrCIP}
\end{figure}

Figure~\ref{reerrCIP2} illustrates the relative error of the CIP-FEM solution for general
$\ga$ other than $\ga_o$ and $-1/12$, say, $\ga=-0.08$ and $\ga=-0.1 \i$, for $k$ from 1 to 1000 on meshes determined by
$k^3h^2=1$. It is shown that the relative error can be controlled.
For small $k$ ($1\leq k\leq 50$), the relative error decreases
rapidly with $k$, for large $k$ ($k\geq100$), the relative error is
dominated by the term $k^3h^2$. It verifies the estimates given
by \eqref{error1} in Theorem~\ref{terr3} and Theorem~\ref{imge th}.
The pollution effect does exist for the two choices of $\ga$.

\begin{figure}[!htb]
\begin{center}
\includegraphics[scale=0.6]{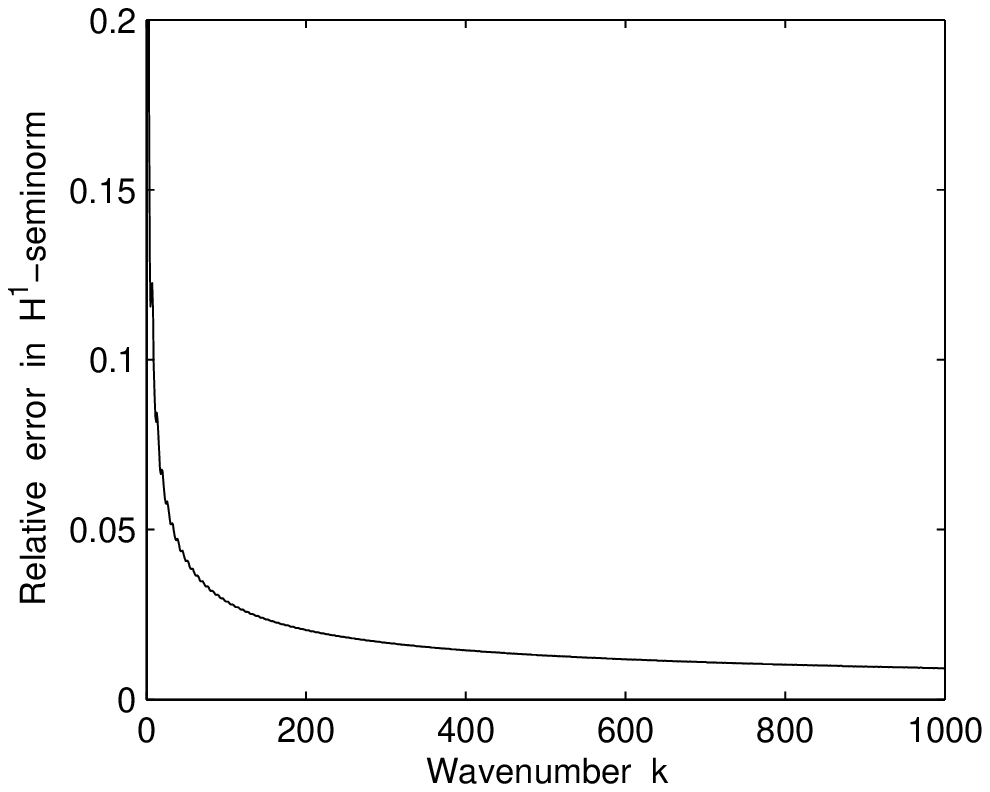}
\includegraphics[scale=0.6]{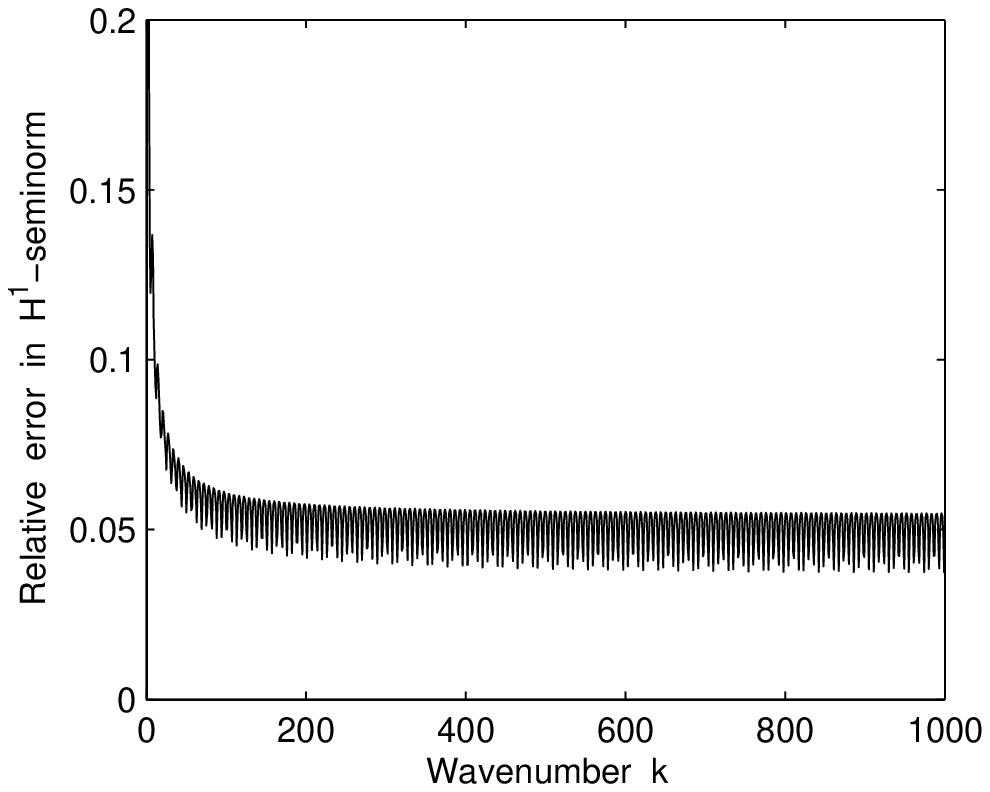}
\end{center}
\caption{\small{The relative error of the CIP-FEM solution with $\ga=-0.08$ (left) and $\ga= -0.1 \i$ (right) in
$H^1$-seminorm with constraint $k^3h^2=1$ for $k$ from 1 to 1000.}}\label{reerrCIP2}
\end{figure}
In Figure~\ref{relerr12}, the ratio $e_{c}/e_{ba}$ computed with the restriction $kh=1$, is plotted for $k$ from 1 to 1000. Obviously, the ratio (in the left of Figure~\ref{relerr12}) is increasing with $k$ on the line. We remark that the ratio line in the right of Figure~\ref{relerr12} is increasing with $k$ and converges to a constant. This is due to that the relative error of the CIP-FEM solution with $\ga$ (a pure imaginary number with negative imaginary part) is bounded at any range by the magnitudes of $\min \{1, k^3h^2\}$ and $kh$ (cf. Theorem~\ref{imge th}). For large $k$ ($k\geq100$), the ratio $e_{c}/e_{ba}\lesssim 1+\min\{1,k^3h^2\}/{kh}=1+\min\{1,k\}$ (for $kh=1$), i.e., the ratio $e_{c}/e_{ba}\leq C$. This shows that the imaginary part of the stabilization gives control of the amplitude of the wave.
\begin{figure}[!htb]
\begin{center}
\includegraphics[scale=0.6]{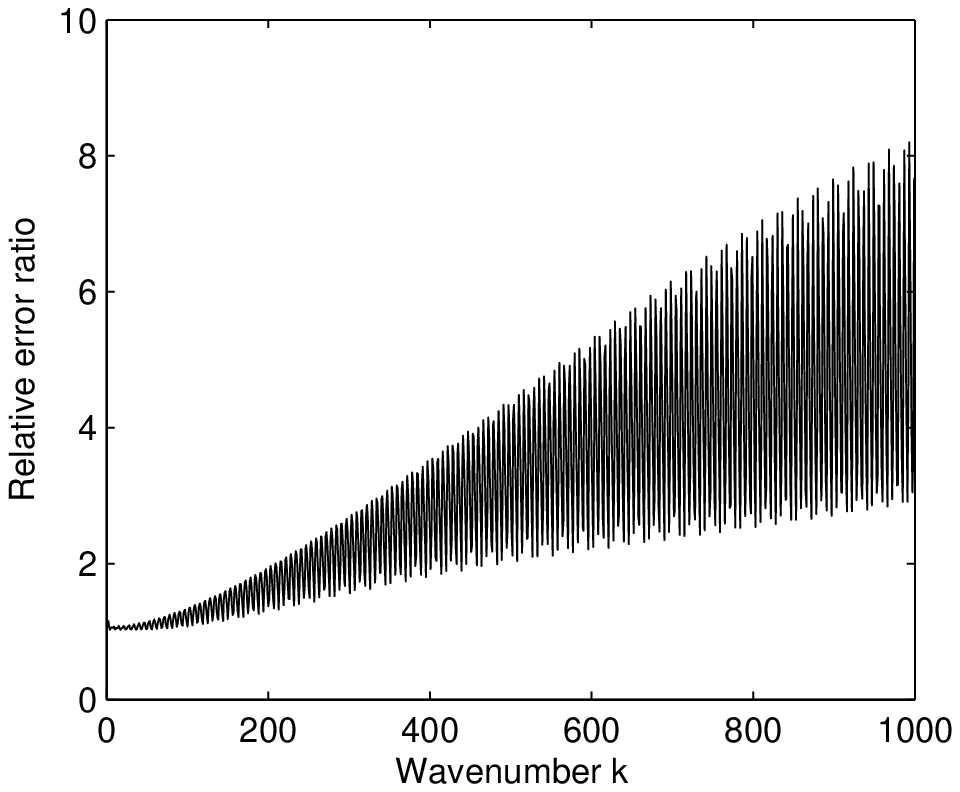}
\includegraphics[scale=0.6]{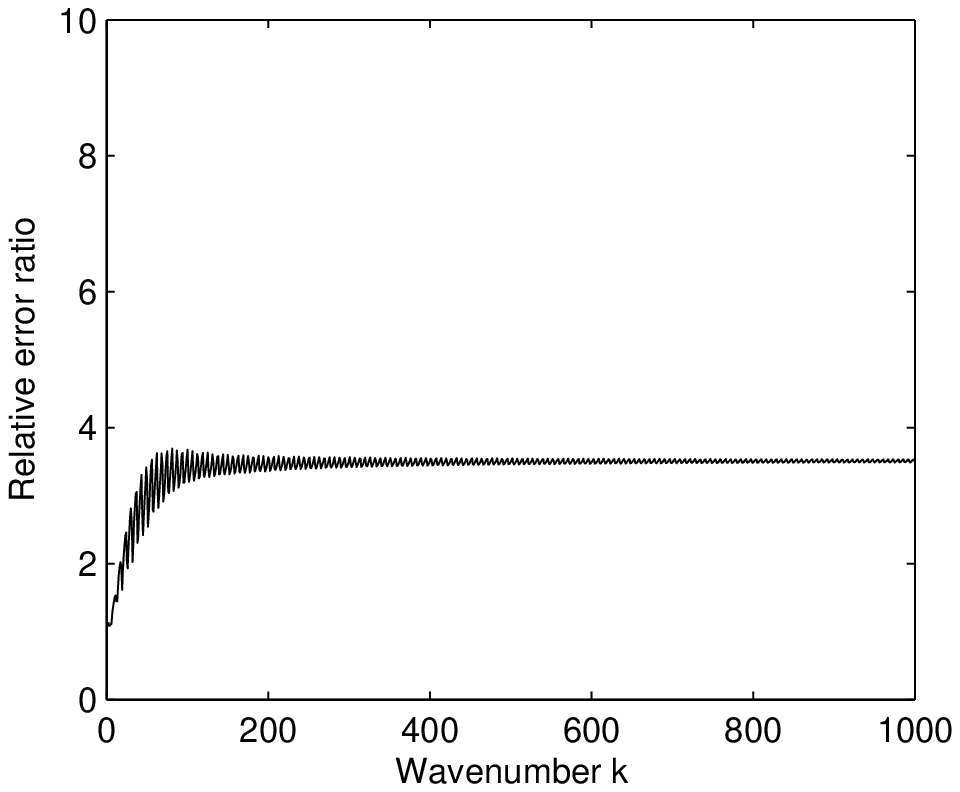}
\end{center}
\caption{\small{The relative error ratio $e_{c}/e_{ba}$ of the CIP-FEM
solution with $\ga=-0.08$ (left) and $\ga=-0.1 \i$ (right) to the
minimal error $H^1$-seminorm with constraint
$kh=1$.}}\label{relerr12}
\end{figure}

\subsection{Eliminate the pollution error} From Figure~\ref{reerrCIP2} and Figure~\ref{relerr12}, we know that the
pollution error is present for general $\ga$, but Figure
\ref{relerropt} shows that the relative error ratio is controlled by
the magnitude $kh$ when we choose an appropriate parameter, say
$\ga=\ga_o$, for $n=k$ up to 1000. The line does neither increase
nor decrease significantly with the change of $k$.
\begin{figure}[!htb]
\begin{center}
\includegraphics[scale=0.6]{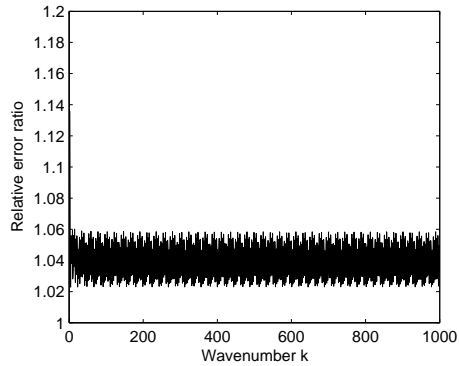}
\end{center}
\caption{\small{The relative error ratio $e_{c}/e_{ba}$ of the CIP-FEM solution with $\ga=\ga_o$ to the minimal error $H^1$-seminorm with constraint $kh=1$.}}\label{relerropt}
\end{figure}

\section{Conclusion} This paper provides some work for analyzing the dispersion and error of CIP method. We have show the following:
\begin{enumerate}
\item The CIP method guarantees existence and amplitude control for
  properly chosen sign of the imaginary part of the stabilization operator.
\item There is numerical pollution for general $\ga$ and the error is mainly influenced by the pollution term for large $k$.
\item There are many possible ``good" choices of parameters to
  eliminate the pollution term. Indeed, provided $kh\leq 1$ the stabilization parameter may
  be chosen in an $O(h)$ interval of the ideal value $\ga_o$.
\end{enumerate}
Future work will address the questions to what extent these results
can be made to carry over to the multidimensional case and to higher
polynomial orders.


\end{document}